%% file: ham-def.tex
\pdfoutput=1
\documentclass[11pt]{article}
\input{preamble}

\newcommand{\mat}[1]{\begin{pmatrix}#1\end{pmatrix}}
\newcommand{\OO}{\mathcal{O}}
\newcommand{\Hh}{\mathcal{H}}
\newcommand{\JJ}{\mathbb J}

\newcommand{\GG}{\mathbb G}

\newcommand{\MM}{\mathcal{M}}

\newcommand{\delbar}{\overline\partial}
\newcommand{\eps}{\epsilon}

\newcommand{\ol}{\overline}

\newcommand{\wt}{\widetilde}

\renewcommand{\Im}{\mathrm{Im}}

\newcommand{\BV}{\delta_\sigma}
\renewcommand{\X}{\mathfrak{X}}

\newcommand{\cT}{{T_\CC}^{\!\!*}}

\newcommand{\Tc}{{T_\CC}}
\renewcommand{\aa}{\pi_T}

    \catcode`\_=11\relax
    \newcommand\email[1]{\_email #1\q_nil}
    \def\_email#1@#2\q_nil{%
      \href{mailto:#1@#2}{{\emailfont #1@#2}}
    }
    \newcommand\emailfont{\sffamily}
    \newcommand\emailampersat{{\color{red}\small@}}
    \catcode`\_=8\relax    
\begin{document}

\title{Generalized K\"ahler metrics from Hamiltonian deformations}
\date{}
\author{Marco Gualtieri\thanks{Department of Mathematics, University of Toronto; \email{mgualt@math.toronto.edu}}}

\maketitle
\centerline{\it Dedicated to Nigel Hitchin on the occasion of his seventieth birthday. }
\abstract{
We give a new characterization of generalized K\"ahler structures in terms of their corresponding complex Dirac structures. We then give an alternative proof of Hitchin's partial unobstructedness for holomorphic Poisson structures.  Our main application is to show that there is a corresponding unobstructedness result for arbitrary generalized K\"ahler structures.  That is, we show that any generalized K\"ahler structure may be deformed in such a way that one of its underlying holomorphic Poisson structures remains fixed, while the other deforms via Hitchin's deformation. 
Finally, we indicate a close relationship between this deformation and the notion of a Hamiltonian family of Poisson structures.}
%, and Poisson modules.}
\tableofcontents

\section{Introduction}

Generalized K\"ahler geometry, first described by Gates, Hull, and Ro\v{c}ek in~\cite{MR776369} in terms of a pair of complex structures $(I_+, I_-)$ compatible with the same Riemannian metric, is in some ways similar to Hyperk\"ahler geometry.  One of the most important aspects of hyperk\"ahler geometry is the presence of an underlying holomorphic symplectic structure.  Its counterpart in generalized K\"ahler geometry is the pair of holomorphic Poisson structures described by Hitchin~\cite{MR2217300} underlying any generalized K\"ahler manifold.  In fact, as shown in~\cite{MR3232003}, this structure is accompanied by a rich set of holomorphic structures, including a pair of holomorphic Courant algebroids, each of which contains a pair of complementary holomorphic Dirac structures.  

Using this underlying structure, a construction was introduced in~\cite{MR2681704,MR2371181} whereby a holomorphic Poisson structure on a Fano manifold was used to deform its usual K\"ahler structure into a generalized K\"ahler structure. A feature of this construction is that the resulting pair of complex structures $(I_+, I_-)$ are equivalent to the original one; they are related by a diffeomorphism.  
The stability theorems of Goto~\cite{MR2955201,MR2669364} provided existence results for more general types of deformations than the ones above, for example, a deformation of a Poisson K\"ahler manifold in which the complex structure bifurcates into a pair $(I_+,I_-)$ of inequivalent complex structures.  Inspired by these stability results, Hitchin observed in~\cite{MR3024823} a fundamental property of holomorphic Poisson structures in general: that for any closed $(1,1)$-form $\omega$ and holomorphic Poisson tensor $\sigma$, the class $[\sigma(\omega)]\in H^1(T)$ is unobstructed in the moduli space of complex structures.   

The purpose of this paper is to lift Hitchin's unobstructedness for holomorphic Poisson structures to arbitrary generalized K\"ahler structures.   That is, we show that any generalized K\"ahler structure may be deformed in such a way that one of its underlying holomorphic Poisson structures remains fixed, while the other deforms via Hitchin's deformation.  Our method may be applied to any of the deformations studied in Hitchin's paper, for instance producing a generalized K\"ahler structure where $I_+$ is a Hilbert scheme of points on a Poisson surface, and $I_-$ is the nontrivial deformation of this Hilbert scheme described by Fantechi~\cite{MR1354269}.  But it can also be applied to deform generalized K\"ahler structures which do not fit into the framework of Goto's stability theorems, such as the generalized K\"ahler structures on compact even-dimensional semisimple Lie groups~\cite{MR3232003}.

In order to explain our construction, we must recast Hitchin's deformation and generalized K\"ahler geometry itself purely in terms of Dirac structures~\cite{MR998124} and the action of closed 2-forms on them via B-field gauge transformations.  In Section~\ref{gaugepo} we recall how Poisson structures are described as Dirac structures, and how gauge transformations may be used to deform them, focusing on the crucial differences between the real and holomorphic cases.  In Section~\ref{GKconstruct} we give a new description of generalized K\"ahler geometry in Dirac terms (Proposition~\ref{recastgk}), and provide in Theorem~\ref{diracdef} a method for deforming generalized K\"ahler structures.  To produce the deformation, the theorem requires a real gauge equivalence between holomorphic Poisson structures; we study these next.  In Section~\ref{hitdef} we provide an alternative proof of Hitchin's unobstructedness result, generalizing it slightly in Theorem~\ref{defex} in a way suggested by~\cite{MR3007085}.  In Section~\ref{sevfam}, we observe that the Poisson deformations under consideration actually define \emph{Hamiltonian families} in the sense of \v{S}evera~\cite{MR1978549}. Finally, in Section~\ref{applic}, we give the promised lift of Hitchin's unobstructedness to arbitrary generalized K\"ahler structures (Theorem~\ref{defexreal} and Corollary~\ref{deformkahler}). 

\noindent \emph{Acknowledgements:}
We thank Nigel Hitchin for lengthy discussions on this topic in 2011 and apologize for the delay in publication. We also thank Joey van der Leer Dur\'an for helpful discussions and corrections.  This research is supported by an NSERC Discovery Grant.  

\section{Notation and basic operations on Dirac structures}\label{sums}

Let $M$ be a smooth manifold and $H\in \Omega^3(M)$ a closed 3-form. Recall that a Dirac structure is a maximal isotropic subbundle $L\subset T\oplus T^*$ of the direct sum of the tangent and cotangent bundles of $M$, that is involutive for the $H$-twisted Courant bracket.  Here isotropy refers to the natural split-signature inner product $\IP{X+\xi,Y+\eta}$ on $T\oplus T^*$, whereas the Courant bracket has the expression 
\begin{equation}
[X+\xi,Y+\eta]_H = \mathcal{L}_X(Y+\eta) - i_Yd\xi  + i_Yi_X H.
\end{equation}

The main tool we need from the theory of Dirac structures is the notion of tensor product of Dirac structures from~\cite{MR2811595}, but it will be convenient to use additive notation.  Let $L_1, L_2$ be Dirac structures for the 3-forms $H_1, H_2$ respectively, such that they are transverse in the sense that $\pi_T(L_1)+\pi_T(L_2) = T$, where $\pi_T:T\oplus T^*\to T$ is the projection.  Then their sum is defined by 
\begin{equation}
L_1 + L_2 = \{X+\alpha + \beta \ |\ X+\alpha\in L_1, X+\beta\in L_2\},
\end{equation} 
and is a Dirac structure for the 3-form $H_1+H_2$. 
%If $L_1$ and $L_2$ have pure spinors given by the differential forms $\rho_1$ and $\rho_2$ respectively, then the above sum has pure spinor given by the wedge product $\rho_1\wedge\rho_2$. 

We may also rescale a Dirac structure by a nonzero element $\lambda\in \RR$ (or $\CC$, in the case of complex Dirac structures in $\Tc\oplus \cT$) as follows:
\begin{equation}
\lambda L = \{X + \lambda\alpha \ |\ X+\alpha\in L\}.
\end{equation} 
This is a Dirac structure for the 3-form $\lambda H$.  
%There is a corresponding operation on pure spinors: for any nonzero scalar $\lambda$ and differential form with degree decomposition $\rho = \sum \rho_i$, we define 
%\begin{equation}
%\ul\lambda\rho = \sum_k \lambda^k (\rho_{2k} + \rho_{2k+1}).
%\end{equation}
%This is compatible with the Clifford action in the following sense:
%\begin{equation}
%\ul\lambda  ((X+\alpha)\cdot \rho) = \lambda^{\delta} (X+\lambda\alpha)\cdot \ul\lambda \rho,
%\end{equation}
%where $\delta = 1$ or $0$ if $\rho$ has even or odd parity, respectively. 
%Note that the reversal antiautomorphism of forms $\rho\mapsto \rho^\bot$, defined to be the identity on forms of degree $0, 1$ and satisfying $(\rho_1\wedge\rho_2)^\bot = \rho_2^\bot\wedge\rho_1^\bot$, coincides with the above operation for $\lambda=-1$. That is, we have 
%\begin{equation}
%\rho^\bot = \ul{-1}\rho. 
%\end{equation}
Because it occurs so often, we use the difference $L_2 - L_1$ to denote $(-1)L_1 + L_2$. Note that when $L_1, L_2$ are both integrable for the 3-form $H$, then $L_2 - L_1$ is Dirac for the zero 3-form. 
%Also, if $\rho_1,\rho_2$ are pure spinors associated to $L_1, L_2$, then a pure spinor for the difference is given by 
%\begin{equation}
%\rho_1^\bot\wedge \rho_2.
%\end{equation}

We may use the above operations to study intersections of Dirac structures, because of the following observation:
\begin{lemma}\label{interwed}
Let $L_1$ and $L_2$ be Dirac structures which are transverse in the sense  $\aa(L_1) + \aa(L_2) = T$.  Then the anchor map defines an isomorphism 
\begin{equation}
\xymatrix{\pi|_{L_1\cap L_2}: L_1\cap L_2\ar[r]^-{\cong} & (L_2- L_1)\cap T}.
\end{equation}
\end{lemma}
\begin{proof}
An element of $(L_2-L_1)\cap T$ is a vector $X$ such that there exists $\alpha\in T^*$ such that $X + \alpha \in L_2$ and $X-\alpha\in -L_1$, but then $X+\alpha \in L_1\cap L_2$, showing surjectivity of the above map. But then the kernel is $L\cap \ol L\cap \cT$, which vanishes by the transversality assumption.
\end{proof}
This is most often used in the situation where $L_1$ and $L_2$ are complementary, i.e. $L_1\oplus L_2 = T\oplus T^*$.  In this case,  $L_1\cap L_2 = \{0\}$, and so $(L_2-L_1)\cap T=0$, implying that $L_2-L_1 = \Gamma_P$ is the graph of a (necessarily skew-symmetric) map $P:T^*\to T$. The Dirac structure $\Gamma_P$ is then involutive for $H=0$ precisely when $P\in C^\infty(\wedge^2 T)$ is a Poisson structure. This observation is a fundamental aspect of Dirac geometry~\cite{MR2642360,MR1262213}.  

Finally, we recall that closed 2-forms $B\in\Omega^2(M)$ act as automorphisms of the Courant bracket: the \emph{gauge transformation} 
\begin{equation}\label{bfield}
e^B: X+\xi \mapsto X + \xi + BX,  
\end{equation}
where $BX = i_X B$, preserves both the inner product $\IP{\cdot,\cdot}$ and the bracket $[\cdot,\cdot]_H$, and so there is an action of closed 2-forms on Dirac structures. A useful observation is that this action may be expressed in terms of the Dirac sum operation; for any Dirac structure $L$, 
\begin{equation}
e^B(L) = L + \Gamma_B,
\end{equation}
where $\Gamma_B = \{X + BX\ :\ X\in T\}$ is the graph of the 2-form $B$; since $dB=0$, $\Gamma_B$ is a Dirac structure for $H=0$.  

The action of gauge transformations on Dirac structures is fundamental for all that follows, as it can be used to produce deformations of any geometric structure which can be defined in terms of Dirac structures.

\section{Gauge transformations of Poisson structures}\label{gaugepo}

In this paper, we consistently view Poisson structures as Dirac structures.  In the case of a real Poisson structure $\pi\in C^\infty(\wedge^2 T)$ on a real smooth manifold $M$, this means that we encode $\pi$ in terms of its graph subbundle $\Gamma_\pi\subset T\oplus T^*$, defined by
\begin{equation}
\Gamma_\pi = \{\pi\xi + \xi \ |\ \xi\in T^*\}.
\end{equation}
As explained in~\cite{MR998124}, this subbundle $\Gamma_\pi$ is maximal isotropic for the natural split-signature bilinear form on $T\oplus T^*$, and is involutive for the Courant bracket if and only if 
$\pi$ satisfies the Jacobi identity $[\pi,\pi]=0$.  
%As explained in~\cite{MR2811595}, giving the Dirac structure $\Gamma_\pi$ is equivalent to specifying a pure spinor line bundle $K_\pi\subset \Omega^\bullet$ of the differential forms. This relies on the fact that the differential forms are an irreducible Clifford module for the Clifford algebra generated by $T\oplus T^*$. Choosing a volume form $\Omega$ allows us to give a generator for this pure spinor line, namely the differential form\marco{not sure if we need forms at all}
%\begin{equation}
%\rho_\pi = e^\pi\Omega = \Omega + i(\pi)\Omega + \tfrac{1}{2!} i(\pi)^2\Omega  + \cdots.
%\end{equation}
%The Poisson condition then translates to the condition that there exists a vector field $v\in\X^1_M$ such that $d\rho_\pi = i(v)\rho_\pi$. This vector field is then uniquely determined by $\Omega$ and is called its modular vector field.

The situation is different for a holomorphic Poisson structure, that is, a smooth manifold $M$ equipped with a complex structure $I$ and a holomorphic Poisson tensor $\sigma\in H^0(M,\wedge^2 T_{1,0})$ such that $[\sigma,\sigma]=0$.  There are two ways to view it as a Dirac structure: we may view $\sigma$ as a morphism $\cT\to\Tc$ and take its graph $\Gamma_\sigma\subset \Tc\oplus\cT$ as in the real case, or, more interestingly for us, we consider the Dirac structure
\begin{equation}\label{dirsig}
L_{\sigma}= \{ X + \sigma(\zeta) + \zeta\ |\ X\in {T}^{0,1}, \zeta\in T^*_{1,0}\}. 
\end{equation}
The advantage of the second approach is that $L_\sigma$ encodes both the complex structure and the Poisson tensor; indeed, $L_\sigma$ is involutive precisely when both the complex structure and the Poisson tensor satisfy their respective integrability conditions.  
%To describe the pure spinor line corresponding to $L_\sigma$ on an $n$-dimensional complex manifold, we choose a holomorphic volume form $\Omega\in\Omega^{n,0}$ and obtain the following generator:
%\begin{equation}
%\rho_{\sigma} = e^\sigma \Omega = \Omega + i(\sigma)\Omega + \tfrac{1}{2!}i(\sigma)^2\Omega + \cdots,
%\end{equation}
%now a complex differential form generating the line subbundle $K_\sigma\subset \Oc$. 

We may characterize the Dirac structures of the form~\eqref{dirsig} in the following way.
\begin{proposition}\label{lsig}
Any complex Dirac structure $L$ for the untwisted Courant bracket such that $\Tc = (L\cap \Tc)\oplus (\ol L\cap\Tc)$ must be of the form~\eqref{dirsig}, for a holomorphic Poisson structure $(I,\sigma)$.
\end{proposition}
\begin{proof}
Define the complex structure $I$ to be such that $T_{0,1}=L\cap\Tc$.  Then by the maximal isotropic condition, the projection of $L$ to $\cT$ must coincide with the annihilator of $T_{0,1}$, i.e., $T^*_{1,0}$. Then, for $\zeta_1, \zeta_2\in T^*_{1,0}$, choose a lift $\tilde\zeta_1\in L$ of $\zeta_1$, unique up to $T_{0,1}$, and define $\sigma(\zeta_1,\zeta_2) = \langle\tilde\zeta_1, \zeta_2\rangle$, which is independent of the choice of lift. It is then straightforward to verify that $L$ coincides with~\eqref{dirsig}.   
\end{proof}

Two Poisson structures on the manifold $M$ are normally considered to be isomorphic when there is a smooth (or holomorphic, in the complex case) automorphism of $M$ relating them.  By viewing the Poisson structures as Dirac structures, however, we may act on them by Courant symmetries~\eqref{bfield}, enlarging the equivalence relation as follows.  We treat the real and complex situations separately. 
\begin{definition}
Real Poisson structures $\pi_0, \pi_1$ on $M$ are \emph{gauge equivalent} when there is a real closed 2-form $B$ such that 
\begin{equation}\label{gaugereal}
e^B\Gamma_{\pi_0} = \Gamma_{\pi_1}.
\end{equation}
\end{definition}
This equivalence relation, introduced in~\cite{MR2023853}, holds precisely when $(1+B\pi_0)$ is invertible as a bundle endomorphism of $T^*$ and the following identity holds:
\begin{equation}\label{gaugereal2}
\pi_1 = \pi_0 (1+ B\pi_0)^{-1}.
\end{equation}
From this we see that $\pi_0, \pi_1$ share the same image, i.e. define the same singular foliation.  As a result, they share the same kernel, which must be invariant under $(1+B\pi_0)$.  Equation~\eqref{gaugereal2} may therefore be restricted to any symplectic leaf, where $\pi_0, \pi_1$ are invertible, defining symplectic forms $\omega_0, \omega_1$, which then satisfy $\omega_1 = \omega_0 +  \iota^*B$, where $\iota$ is the inclusion of the leaf in question.  In summary, $\pi_0, \pi_1$ are gauge equivalent precisely when their symplectic leaves coincide, with symplectic forms differing by a the restriction of a global closed form. 

\begin{definition}\label{gaugcxp}
Holomorphic Poisson structures $(I_0,\sigma_0)$ and $(I_1,\sigma_1)$ are \emph{gauge equivalent} when there is a complex closed 2-form $\beta\in\Omega^2_\CC$ such that 
\begin{equation}\label{cxb}
e^\beta L_{\sigma_0} = L_{\sigma_1}.
\end{equation} 
\end{definition}
The simplicity of the above condition is somewhat deceptive, especially because of the fact that the underlying complex structures $I_0$ and $I_1$ are different as tensors and possibly even give rise to inequivalent complex manifolds.  We may analyse~\eqref{cxb} with the following equivalent formulation.  We use the notation $T_{1,0}(I_k)$ and $T_{0,1}(I_k)$ to denote the holomorphic and antiholomorphic tangent bundle, respectively, of the complex structure $I_k$, for $k=0,1$. 
\begin{proposition}\label{complexb}
The holomorphic Poisson structures $(I_0,\sigma_0)$ and $(I_1,\sigma_1)$ are gauge equivalent in the above sense if and only if there is a complex closed 2-form $\beta\in\Omega^2_\CC$ satisfying all of the following conditions:
\begin{itemize}
\item[i)] $\beta(T_{0,1}(I_0)) \subset T^*_{1,0}(I_1)$,
\item[ii)] $(1-\sigma_1\beta)(T_{0,1} (I_0)) \subset T_{0,1}(I_1)$,
\item[iii)] $(1+\sigma_0\beta)(T_{0,1}(I_1))\subset T_{0,1}(I_0)$,
\item[iv)] $(\sigma_1-\sigma_0 + \sigma_1\beta\sigma_0)(T^*_{1,0}(I_0))\subset T_{0,1}(I_1)$.
\end{itemize}
\end{proposition}
\begin{proof}
We simply express the containment $e^\beta L_{\sigma_0}\subseteq L_{\sigma_1}$ explicitly. By definition, $X+\beta X\in e^\beta L_{\sigma_0}$ for all $X\in T_{0,1}(I_0)$, and this is contained in $L_{\sigma_1}$ if and only if, first, $\beta X$ lies in $T^*_{1,0}(I_1)$, giving condition $i)$, and second, that $X - \sigma_1\beta X$ lies in $T_{0,1}(I_1)$, giving condition $ii)$.  Also from the definition, $\sigma_0\xi + \xi + \beta\sigma_0\xi$ lies in $e^\beta L_{\sigma_0}$ for all $\xi\in T^*_{1,0}(I_0)$.  This section lies in $L_{\sigma_1}$ precisely when, first, $(1+\beta\sigma_0)\xi$ lies in $T^*_{1,0}(I_1)$, which is equivalent to condition $iii)$ by duality, and second, that  $\sigma_0\xi - \sigma_1(1+\beta\sigma_0)\xi$ lies in $T_{0,1}(I_1)$, which is condition $iv)$. Since these two subspaces of $e^\beta(L_{\sigma_0})$ span, and since containment of maximal isotropic subspaces implies equality, we obtain the result. 
\end{proof}

An important special case is when $\beta$ is real.  Let $Q_i$ be the imaginary part of $4\sigma_i$ for $i=0,1$, so that $\sigma_i = \tfrac{1}{4}(I_iQ_i + i Q_i)$.  Then conditions $i)$--$iv)$ above are cumulatively equivalent to the following conditions:
\begin{itemize}
\item[\itshape i)] $\beta I_0 + I_1^* \beta =0$,
\item[\itshape ii)] $I_0-I_1 = Q_1\beta$,
\item[\itshape iii)] $I_0-I_1 = Q_0\beta$,
\item[\itshape iv)] $Q_0=Q_1$.
\end{itemize}
Setting $F=\beta$ and $Q=Q_0=Q_1$, we obtain the following simplified conditions, first studied in~\cite{MR2681704}:
\begin{proposition}\label{realcxgauge}
The holomorphic Poisson structures $(I_0,\sigma_0)$ and $(I_1,\sigma_1)$ are gauge equivalent via a real closed 2-form $F\in\Omega^2_\RR$ if and only if they share an imaginary part, i.e. $Q = \Im(4\sigma_0) = \Im(4\sigma_1)$, and the following conditions hold:
\begin{equation}\label{splitcon}
\begin{aligned}
FI_0 + I_1^*F &= 0,\\
I_0-I_1 &= QF.
\end{aligned}
\end{equation}
\end{proposition}

\begin{remarks}\mbox{}
\begin{enumerate}
\item Proposition~\ref{realcxgauge} may be rephrased in the following way:  if we fix a holomorphic Poisson structure $(I_0,\sigma_0)$, then the gauge transformation of $L_{\sigma_0}$ by a real closed 2-form $F\in\Omega^2_\RR$ is holomorphic Poisson if and only if 
\begin{equation}
FI_0 + I_0^*F = FQF.
\end{equation}
This condition is obtained from the conditions~\eqref{splitcon} by eliminating $I_1$, but if it is satisfied, then we may define $I_1 = I_0 - QF$, obtaining a complex structure which satisfies~\eqref{splitcon}.

\item The importance of this special case (i.e. a real gauge transformation of a complex Dirac structure) for our purposes is that it can be used to construct generalized K\"ahler structures.  The construction given in~\cite{MR2681704}, which we shall simplify and generalize in Section~\ref{GKconstruct}, produces a generalized K\"ahler structure in which one of the two generalized complex structures is symplectic, with the above 2-form $F$ playing the role of the symplectic form.

\item The crucial difference between gauge transformations of real and holomorphic Poisson structures is as follows.  In the real case, given a Poisson structure $\pi_0$, any closed 2-form $B\in\Omega^2_\RR$ such that $(1+B\pi_0)$ is invertible may be applied to $\pi_0$, deforming it to the new Poisson structure $\pi_1$ given by Equation~\eqref{gaugereal2}. Indeed, any sufficiently small closed 2-form satisfies this condition and so it is easy to deform Poisson structures in this way.  On the other hand, in the holomorphic case, the conditions of Proposition~\ref{complexb} impose nonlinear algebraic constraints on the closed 2-form $\beta$.  In terms of the foliation by holomorphic symplectic leaves, we are shifting the leafwise symplectic forms by $\beta$ but these must satisfy algebraic constraints in order to remain holomorphic symplectic forms.  As a result, we require further techniques in order to solve these conditions.  One such technique is provided in~\cite{MR2681704}, and in Section~\ref{applic} we provide a new, more general method which is inspired by Hitchin's unobstructedness result for holomorphic Poisson structures~\cite{MR3024823}.
\end{enumerate}
\end{remarks}

\section{Generalized complex and K\"ahler structures}\label{GKconstruct}

\subsection{An unobstructedness for generalized complex structures}

A generalized complex structure is a complex structure $\JJ$ on $T\oplus T^*$ whose $+i$-eigenbundle $L\subset \Tc\oplus\cT$ is a complex Dirac structure for the $H$--twisted Courant bracket.  Indeed, generalized complex structures may be defined simply as complex Dirac structures $L$ satisfying $L\cap \ol L = 0$.  In general, if we assume the transversality condition $\pi_\Tc(L) + \pi_\Tc(\ol L) = \Tc$, then by Lemma~\ref{interwed}, the intersection $L\cap\ol L$ projects isomorphically to $(L-\ol L)\cap \Tc$, so the complex Dirac structure $L$ defines a generalized complex structure if and only if $L\cap\ol L\cap \cT = 0$ and the following real Dirac structure,
\begin{equation}\label{pill}
\Gamma_\pi = \tfrac{1}{2i}( L - \ol L),
\end{equation}
has trivial intersection with $T$, i.e. defines a real Poisson structure $\pi$. This Poisson structure is fundamental for the study of generalized complex structures: in a sense made precise in~\cite{MR2811595,MR2240211,bailey-2012,BLMS:BLMS12029}, the generalized complex structure may be viewed as a holomorphic structure transverse to the symplectic leaves of $\pi$.  

This basic observation leads us to the following result concerning the ability to deform generalized complex structures by complex gauge transformations. 
\begin{proposition}\label{gcdef}
Fix a generalized complex structure $\JJ_0$, with $+i$ eigenbundle $L_0$ and underlying real Poisson structure $\pi_0$. Let $\beta\in\Omega^2_\CC$ be a complex closed 2-form. Then $L_1=e^\beta L_0$ defines a  generalized complex structure $\JJ_1$ if and only if $(1+B\pi_0)$ is invertible, for $B = \Im(\beta)$. The underlying real Poisson structure $\pi_1$ of $\JJ_1$ is then given by gauge transformation of $\pi_0$ by $B$, as in \eqref{gaugereal2}.
\end{proposition}
\begin{proof}
We may express $L_1$ as a sum of Dirac structures in the sense of Section~\ref{sums}: if $\Gamma_\beta$ is the graph of $\beta$, then	
\begin{equation}
L_1 = e^\beta(L_0) = L_0 + \Gamma_\beta.
\end{equation}
The Dirac structure $L_1$ defines a generalized complex structure if and only if $L_1\cap \ol L_1 = 0$, which by Lemma~\ref{interwed} is equivalent to the condition that 
\begin{equation}
\tfrac{1}{2i}(L_1 -\ol L_1) = \tfrac{1}{2 i}(L_0 + \Gamma_\beta - \ol L_0 - \Gamma_{\ol\beta})  = \Gamma_{\pi_0} + \Gamma_{\Im\beta} 
\end{equation}
is a Poisson structure.  But this is precisely the gauge transformation $e^B\Gamma_{\pi_0}$, which is Poisson if and only if $(1+B\pi_0)$ is invertible.   Under this condition, Equation~\ref{gaugereal2} gives the expression for this Poisson structure, as required.
\end{proof}

The above result may be recast as an unobstructedness result for deformations of generalized complex structures.  As explained in~\cite{MR2811595}, infinitesimal deformations of generalized complex structures up to equivalence are given by the second Lie algebroid cohomology $H^2_L$ of the Dirac structure $L$.  The projection $\aa:L\to \Tc$ induces a pullback morphism $\aa^*:H^2(M,\CC)\to H^2_L$, with the following property.

\begin{corollary}
On a compact manifold, any infinitesimal deformation of the generalized complex structure $\JJ$ which lies in the image of the pullback map $\aa^*:H^2(M,\CC)\to H^2_L$ is unobstructed.  
\end{corollary}
\begin{proof}
For any class $[\beta]\in H^2(M,\CC)$, choose a representative $\beta\in\Omega^2_\CC$ and define the family $\{L_t = e^{t\beta}L\}_{t\in\CC}$ of complex Dirac structures.  If $M$ is compact, then for $t$ sufficiently small the condition of Proposition~\ref{gcdef} is satisfied, implying $L_t\cap\ol L_t=0$, so that we obtain a family of generalized complex structures parametrized by a neighbourhood of $0\in\CC$.  Taking the derivative at $t=0$, we obtain the required deformation direction
\begin{equation}
[\tfrac{d}{dt} L_t]\big\rvert_{t=0} = \aa^*([\beta]).
\end{equation}
\end{proof}

\subsection{An unobstructedness for generalized K\"ahler structures}

We now describe the main idea of the paper: an analog of the above unobstructedness result for generalized K\"ahler structures.  A generalized K\"ahler structure is a pair $(\JJ_1,\JJ_2)$ of commuting generalized complex structures such that $\GG=-\JJ_1\JJ_2$ is positive-definite, in the sense that $\IP{\GG u,u}> 0$ for all $u\neq 0$ in $T\oplus T^*$.

Our first task is to recast the generalized K\"ahler condition on $(\JJ_1,\JJ_2)$ in terms of the corresponding pair of complex Dirac structures $(L_1, L_2)$.  The commutation condition $\JJ_1\JJ_2=\JJ_2\JJ_1$ is equivalent to the fact that each $\JJ_1$ eigenbundle decomposes into a direct sum of $\JJ_2$ eigenbundles, that is,
\begin{equation}
L_1 = (L_1\cap L_2) \oplus (L_1\cap \ol L_2) = \ell_+ \oplus \ell_-,
\end{equation}  
where we define $\ell_+ = L_1\cap L_2$ and $\ell_-=L_1\cap \ol L_2$.  
As a result we obtain the following decomposition:
\begin{equation}
\Tc\oplus \cT = \ell_+\oplus \ell_-\oplus \ol\ell_+\oplus\ol\ell_-.
\end{equation}
The positive-definiteness of $\GG$ is then equivalent to the condition that $\ell_+\oplus\ol\ell_+ = V_+\otimes\CC$ for a positive-definite subspace of $V_+\subset T\oplus T^*$. 

\begin{proposition}\label{recastgk}
The pair of complex Dirac structures $(L_1, L_2)$ defines a generalized K\"ahler structure if and only if it satisfies all of the following conditions:
\begin{itemize}
\item[i)] $L_1$ is transverse to $\ol L_1$, i.e., $\pi_{\Tc}(L_1) + \pi_{\Tc}(\ol L_1) = \Tc$, and similarly for $L_2$. 
\item[ii)] The real Dirac structures 
\begin{equation}\label{realpoiss}
\Gamma_{\pi_1} = \tfrac{1}{2i}(L_1-\ol L_1),\qquad \Gamma_{\pi_2} = \tfrac{1}{2i}(L_2-\ol L_2),
\end{equation}
define real Poisson structures, i.e. $\Gamma_{\pi_1}\cap T=\Gamma_{\pi_2}\cap T = 0$.

\item[iii)] The complex Dirac structures 
\begin{equation}\label{holpoiss}
L_{\sigma_+} = \tfrac{1}{2i}(L_1 - L_2),\qquad L_{\sigma_-} = \tfrac{1}{2i}(L_1 - \ol L_2),
\end{equation}
define holomorphic Poisson structures $(I_+,\sigma_+)$, $(I_-,\sigma_-)$ respectively, i.e. $\Tc = (L_{\sigma_+}\cap\Tc)\oplus (\ol L_{\sigma_+}\cap\Tc)$ and $\Tc = (L_{\sigma_-}\cap\Tc)\oplus (\ol L_{\sigma_-}\cap\Tc)$. 
\item[iv)] For all nonzero $u\in L_1\cap L_2$, we have $\IP{u,\ol u}>0$. 
\end{itemize}
\end{proposition}
\begin{proof}
We first show that conditions \emph{i)-iv)} imply the generalized K\"ahler conditions. 
By Lemma~\ref{interwed}, conditions \emph{i)} and \emph{ii)} are equivalent to the conditions $L_1\cap \ol L_1 = 0$ and $L_2\cap \ol L_2 = 0$, which are necessary and sufficient for $L_1, L_2$ to define a pair $(\JJ_1,\JJ_2)$ of generalized complex structures.  By Proposition~\ref{lsig}, the Dirac sum $\tfrac{1}{2i}(L_1-L_2)$ is holomorphic Poisson if and only if its intersection with $\Tc$ is the $-i$ eigenspace of a complex structure.  Lemma~\ref{interwed} then identifies this intersection with $\ell_+ = L_1\cap L_2$, which must therefore have half the rank of $L_1$.  Repeating this argument with $\tfrac{1}{2i}(L_1-\ol L_2)$, we obtain another half-rank subbundle 
$\ell_- = L_1\cap\ol L_2$.  The condition $L_2\cap\ol L_2 = 0$ then implies that $L_1 = \ell_+\oplus\ell_-$, so that conditions \emph{i)}, \emph{ii)} and \emph{iii)} imply that $(L_1,L_2)$ define a commuting pair of generalized complex structures $(\JJ_1,\JJ_2)$.  The final condition \emph{iv)} then implies that the real subbundle $V_+\subset T\oplus T^*$ defined by $V_+\otimes\CC = \ell_+\oplus\ol\ell_+$ is positive definite, and since $\ell_+$ has half the rank of $\Tc$, this means that $V_+$ is a maximal positive-definite subbundle of $T\oplus T^*$, which finally implies that $\GG = -\JJ_1\JJ_2$ is positive-definite, as required.  
The converse, i.e. the fact that a generalized K\"ahler structure satisfies the conditions \emph{i)--iv)}, is straightforward and shown in~\cite{MR3232003}.
\end{proof}

\begin{remarks}\mbox{}\label{remarksred}
\begin{enumerate}
\item As explained in~\cite{MR3232003}, the holomorphic Poisson structures occurring in~\eqref{holpoiss} coincide (up to a $2i$ rescaling which will be convenient for us below) with the 
holomorphic Poisson structures described by Hitchin~\cite{MR2217300}, which play a central role in generalized K\"ahler geometry.  The observation of Proposition~\ref{recastgk} is that these Poisson structures encode the fact that $\JJ_1, \JJ_2$ commute. 
\item If conditions~\eqref{holpoiss} are satisfied pointwise, they imply the transversality of the pairs $(L_1,\ol L_1)$, $(L_2,\ol L_2)$, $(L_1,L_2)$, and $(L_1,\ol L_2)$ relative to $\Tc$, so there is no question that the Dirac sums are well-defined as Dirac structures. Furthermore, this shows that condition \emph{i)} is unnecessary in the Theorem.   
\item Expressions~\eqref{realpoiss} and \eqref{holpoiss} make several relationships between the real and holomorphic Poisson structures transparent.  For example, \eqref{holpoiss} implies that we have the coincidence of imaginary parts $\Im(\sigma_+) =\Im(\sigma_-)$, due to the following equality of Dirac structures:
\begin{equation}\label{impartsigma}
\Gamma_{\Im(\sigma_+)} = \tfrac{1}{2i}(L_{\sigma_+} - \ol L_{\sigma_+}) = \tfrac{1}{2i}(L_{\sigma_-} - \ol L_{\sigma_-}) = \Gamma_{\Im(\sigma_-)}.
\end{equation}

\item Conditions \emph{i)}, \emph{ii)}, and \emph{iv)} are open conditions and so remain true under small deformations, whereas \emph{iii)} is not.  Assuming \emph{iii)} holds, condition \emph{ii)} gives rise to a nondegenerate (but possibly indefinite) generalized K\"ahler metric.  Condition \emph{iv)} is needed for the positive-definiteness of this metric.  It is often useful to consider \emph{degenerate} generalized K\"ahler structures, where \emph{iii)} alone holds.
\end{enumerate}
\end{remarks}
\begin{definition}
A degenerate generalized K\"ahler structure is a pair $(L_1,L_2)$  of complex Dirac structures such that $L_{\sigma_\pm}$, as in~\eqref{holpoiss}, define holomorphic Poisson structures.
\end{definition}

We now extend Proposition~\ref{gcdef} to the generalized K\"ahler setting, where we have a pair $(L_1,L_2)$ satisfying the conditions of Proposition~\ref{recastgk}.   Let $\beta_1, \beta_2$ be a pair of closed complex 2-forms, defining deformations
\begin{equation}\label{gaugel12}
L_1' = e^{\beta_1} L_1\qquad L_2' = e^{\beta_2} L_2.
\end{equation}
The main question is: when do these deformed Dirac structures satisfy the conditions of Proposition~\ref{recastgk}, thereby  defining a deformation of generalized K\"ahler structure?
For sufficiently small $\beta_1,\beta_2$, conditions \emph{i)} and \emph{ii)} remain true.  We must check condition \emph{iii)}:
\begin{equation}\label{calcdef}
\begin{aligned}
\tfrac{1}{2i}(L_1' - L_2') &= \tfrac{1}{2i}(L_1 - L_2) + (\Gamma_{\beta_1/2i} - \Gamma_{\beta_2/2i}) = e^{\tfrac{1}{2i}(\beta_1-\beta_2)} L_{\sigma_+},\\
\tfrac{1}{2i}(L_1' - \ol L_2') &= \tfrac{1}{2i}(L_1 - \ol L_2) + (\Gamma_{\beta_1/2i} - \Gamma_{\ol \beta_2/2i}) = e^{\tfrac{1}{2i}(\beta_1-\ol\beta_2)} L_{\sigma_-}.
\end{aligned}
\end{equation}
From this calculation we see how the complex gauge transformations~\eqref{gaugel12} of $L_1, L_2$ give rise to corresponding complex gauge transformations of $L_{\sigma_+}$ and $L_{\sigma_-}$, by $\beta_+ = \tfrac{1}{2i}(\beta_1-\beta_2)$ and $\beta_- = \tfrac{1}{2i}(\beta_1-\ol\beta_2)$, respectively.  The map $(\beta_1,\beta_2)\to (\beta_+,\beta_-)$ is not an isomorphism; indeed, $\beta_\pm$ have the same imaginary parts, as expected from~\eqref{impartsigma}, and are not affected by shifting both $\beta_1,\beta_2$ by a real 2-form. Therefore, given $\beta_\pm$ with coincident imaginary parts, we may lift these uniquely to $(\beta_1,\beta_2)$ with opposite real parts, giving the following result.

\begin{proposition}\label{defgk}
Let $(L_1,L_2)$ be a degenerate generalized K\"ahler structure, i.e. a pair of complex Dirac structures such that $L_{\sigma_+}, L_{\sigma_-}$, defined by~\eqref{holpoiss}, define holomorphic Poisson structures.  Let $\beta_\pm = F_\pm + iB$ be closed 2-forms with coincident imaginary parts that define new holomorphic Poisson structures via the gauge transformations
\begin{equation}\label{givendef}
L'_{\sigma_+} = e^{\beta_+} L_{\sigma_+},\qquad L'_{\sigma_-} = e^{\beta_-} L_{\sigma_-}.
\end{equation}
Then this deformation lifts to the generalized K\"ahler structure: defining $\beta_1 = -B + i(F_- + F_+)$ and $\beta_2 = B + i(F_- - F_+)$, the pair 
\begin{equation}
L_1' = e^{\beta_1}L_1,\qquad L_2' = e^{\beta_2}L_2
\end{equation}
defines a degenerate generalized K\"ahler structure with underlying holomorphic Poisson structures~\eqref{givendef}.
\end{proposition}

We present two applications of this result: first we show how to recover the construction of generalized K\"ahler structures in~\cite{MR2681704}, and then we present the more general deformation result. 

\begin{example}\label{flocon}
Let $(I,\sigma)$ be a holomorphic Poisson structure.  Consider the following smooth family of pairs of complex Dirac structures parametrized by $t\in\RR$:
\begin{equation}
L_1(t) =2i L_{t\sigma},\qquad L_2(t) = \Tc.
\end{equation} 
This is a family of degenerate generalized  K\"ahler structures, since  
\begin{equation}
L_{\sigma_+} = \tfrac{1}{2i}(L_1-L_2) = L_{t\sigma},\qquad L_{\sigma_-} = \tfrac{1}{2i}(L_1-\ol L_2) = L_{t\sigma}
\end{equation}
is indeed a pair of holomorphic Poisson structures.  

We now deform this pair using the method described in~\cite{MR2681704}, as follows.  Let $\omega\in\Omega^{1,1}(M,\RR)$ and $V$ a real vector field satisfying the system
\begin{equation}\label{systcrvreal}
\begin{aligned}
\mathcal{L}_V I &= Q \omega,\\
\mathcal{L}_V Q &=0,
\end{aligned}
\end{equation}
for $Q = \Im(4\sigma)$ .
%(see Appendix~\ref{pmod} for a discussion of this system)  
For example, for any $f\in C^\infty(M,\RR)$  we may take $\omega = dd^c f$ and $V=-Qdf$ to be the associated Hamiltonian vector field.   
Let $\varphi_t$ be the time-$t$ flow of $V$, and  
\begin{equation}
\wt F(t) = \int_0^t \varphi_s(\omega) ds.
\end{equation}
As explained in~\cite{MR2681704},
%Then by Proposition~\ref{flowg}, 
this
defines a gauge transformation between $(I,\sigma)$ on the one hand and $(I_t, \sigma_t)=(\varphi_t(I), \varphi_t(\sigma))$ on the other.  In terms of Dirac sums, we have 
\begin{equation}
L_{\sigma} + \Gamma_{\wt F(t)} = L_{\sigma_t}.
\end{equation}  
Since $\wt F(0)=0$, we may define $F(t) = t^{-1}\wt F(t)$, and if we scale the above Dirac structures by $t^{-1}$, we obtain
\begin{equation}
L_{t\sigma} + \Gamma_{F(t)} = L_{t\sigma_t}.
\end{equation}  
Finally, applying Proposition~\ref{defgk} with $B = 0$, $F_+ = 0$, and $F_- = F(t)$, we obtain the following family of degenerate generalized K\"ahler structures
\begin{equation}
L_1'(t) = e^{iF(t)} L_{t\sigma/2i},\qquad L_2'(t) = e^{iF(t)}L_2(t) = \Gamma_{iF(t)}.
\end{equation}
Now observe that $L_1'(0) = T_{0,1}\oplus T^*_{1,0}$ and $F(0) = \omega$; if we choose $\omega$ to be a K\"ahler form, then $(L'_1(0), L'_2(0))$ is a genuine K\"ahler structure.  As a result, for sufficiently small $t$, the open conditions \emph{i)}, \emph{ii)} and \emph{iv)} of Proposition~\ref{recastgk} continue to hold, defining a deformation of generalized K\"ahler structure.  By construction, the corresponding holomorphic Poisson structures are then given by 
$L_{\sigma_+}' = L_{t\sigma}$ and $L_{\sigma_-}' = L_{t\sigma_t}$.
That is, $(I_+,\sigma_+) = (I, t\sigma)$ remains unchanged while $(I_-(t),\sigma_-(t)) = (\varphi_t(I), t\varphi_t(\sigma))$ is given by the flow $\varphi_t$. 
\end{example}
The above construction, in which a K\"ahler structure is deformed to a generalized K\"ahler structure in such a way that the complex structures $I_\pm$ are different but related by a diffeomorphism, is at the heart of the constructions given in~\cite{MR2371181} and~\cite{MR2681704}.  
It relies on finding a solution to the system~\eqref{systcrvreal}; %which has a cohomological interpretation explained in Appendix~\ref{pmod}. 
in particular, the first equation $\mathcal{L}_VI=Q\omega$ may be solved if and only if the cup product $\sigma\cdot[\omega]\in H^1(T)$ vanishes.  
As explained in~\cite{MR2681704}, a solution to~\eqref{systcrvreal} is provided by any rank one Poisson module over a holomorphic Poisson manifold;  
%an interesting example 
%of a class $[\omega]\in H^1(\Omega^1)$ which satisfies these conditions is the Atiyah class of a rank 1 Poisson module.  
Because the anticanonical line bundle of a Poisson manifold is always a Poisson module, we see that the above construction may be implemented on any Poisson Fano variety.  

Taking inspiration from the previous example, we now explain how an arbitrary generalized K\"ahler structure may be deformed, assuming that one is given a gauge transformation of one of the underlying holomorphic Poisson structures, as in Proposition~\ref{realcxgauge}.  Importantly, we will not assume that this gauge transformation derives from a solution to the system~\eqref{systcrvreal}; later we will see that it may instead involve a class $[\omega]$ for which $\sigma\cdot[\omega]\in H^1(T)$ is nonzero, indicating a nontrivial deformation of the underlying complex structure.  

%\begin{theorem}%\label{diracdef}
%Let $(L_1,L_2)$ be a pair of complex Dirac structures defining a generalized K\"ahler structure, i.e., satisfying the conditions of Proposition~\ref{recastgk}, and let $\{F_t\in \Omega^2_\RR\}_{t\in\RR}$ be a smooth family of real closed 2-forms with $F_0=0$ and satisfying the condition  
%\begin{equation}
%F_tI_+ + I_+^*F_t = F_tQF_t.
%\end{equation}
%%therefore defining a family of gauge-equivalent holomorphic Poisson structures $e^{F_t} L_{\sigma_+}$. 
%Then the deformation 
%\begin{equation}\label{thedef}
%(L_1'(t), L_2'(t)) = (e^{iF_t}L_1, e^{-iF_t}L_2)
%\end{equation}
%is generalized K\"ahler for $t$ sufficiently small, with corresponding holomorphic Poisson structures 
%\begin{equation}
%L'_{\sigma_+}(t) = e^{F_t} L_{\sigma_+},\qquad L'_{\sigma_-}(t) = L_{\sigma_-}.
%\end{equation}
%Furthermore, if $\pi_1,\pi_2$ are the real Poisson structures underlying $L_1,L_2$, then the family~\eqref{thedef} is generalized K\"ahler in the maximal open interval containing $0$ in which both $(1-F_t\pi_1)$ and $(1-F_t\pi_2)$ are invertible. 
%\end{theorem}

\begin{theorem}\label{diracdef}
Let $(L_1,L_2)$ be the pair of complex Dirac structures defining a generalized K\"ahler structure as in Proposition~\ref{recastgk}, and let $\{F_t\in \Omega^2_\RR\}_{t\in\RR}$ be a smooth family of real closed 2-forms with $F_0=0$ and such that $e^{F_t}L_{\sigma_+}$ is a family of holomorphic Poisson structures. 
Then the deformation 
\begin{equation}\label{thedef}
(L_1'(t), L_2'(t)) = (e^{iF_t}L_1, e^{-iF_t}L_2)
\end{equation}
is generalized K\"ahler for $t$ sufficiently small, with corresponding holomorphic Poisson structures
\begin{equation}
L'_{\sigma_+}(t) = e^{F_t} L_{\sigma_+},\qquad L'_{\sigma_-}(t) = L_{\sigma_-}.
\end{equation}
Furthermore, if $\pi_1,\pi_2$ are the real Poisson structures underlying $L_1,L_2$, then the family~\eqref{thedef} is generalized K\"ahler in the maximal open interval containing $0$ in which both $(1+F_t\pi_1)$ and $(1+F_t\pi_2)$ are invertible. 
\end{theorem}
\begin{proof}
We simply verify conditions \emph{i)--iv)} in Proposition~\ref{recastgk} for the pair $(L_1', L_2')$. In fact, by the second remark following the theorem, $i)$ follows from $ii)$ and $iii)$, so we focus on the latter two. 
Condition \emph{ii)} requires that $\tfrac{1}{2i}(L_1'-\ol L_1')$ is Poisson. But this is given by 
\begin{equation}
\tfrac{1}{2i}(L_1'-\ol L_1') =\tfrac{1}{2i}(L_1-\ol L_1)  + \tfrac{1}{2i}(\Gamma_{iF_t} - \Gamma_{-iF_t}) = e^{F_t}\Gamma_{\pi_1},
\end{equation}
which is Poisson precisely when $(1+F_t\pi_1)$ is invertible. The same condition for $L_2'$ shows that \emph{i)} is satisfied in the required interval around zero.
Condition \emph{iii)} is verified by assumption, in the sense that $F_t$ is assumed to have the property that
\begin{equation}
L'_{\sigma_+} = \tfrac{1}{2i}(L'_1 - L'_2) = L_{\sigma_+} + \tfrac{1}{2i}(\Gamma_{iF_t} - \Gamma_{-iF_t}) = e^{F_t}L_{\sigma_+}
\end{equation}
remains holomorphic Poisson ($L'_{\sigma_-} = L_{\sigma_-}$ remains unchanged). 
Finally, the positivity condition $iv)$: as long as conditions \emph{i), ii), iii)} hold, we have a decomposition
\begin{equation}
\Tc\oplus \cT = \ell_+ \oplus \ell_-\oplus \ol\ell_+\oplus\ol\ell_-,
\end{equation}
where $\ell_+ = L_1' \cap L_2'$ and $\ell_- = L_1'\cap \ol L_2'$.  Since $L_1', L_2'$ are maximal isotropic, $\ell_+^\bot = L_1+L_2 = \ell_+\oplus\ell_-\oplus\ol\ell_-$, and so the pairing $\IP{\cdot,\cdot}:\ell_+\times\ol\ell_+\to\CC$ is perfect, giving a nondegenerate symmetric bilinear form to the real space $V_+$ defined by $V_+\otimes\CC = \ell_+\oplus\ol\ell_+$.  Since we are assuming that this is positive-definite at $t=0$, continuity implies it will be positive-definite on the interval for which \emph{i)}, \emph{ii)} and \emph{iii)} hold, as claimed.
\end{proof}

Theorem~\ref{diracdef} lifts a gauge transformation of the Poisson structure $\sigma$ to a deformation of any generalized K\"ahler structure having $\sigma$ as one of its underlying Poisson structures. In the next section we explain how to obtain such gauge transformations in the first place.   

\section{Deformations of holomorphic Poisson structures}\label{hitdef}

Let $(I,\sigma)$ be a holomorphic Poisson structure on the smooth manifold $M$.  In this section we are interested in the deformation theory of such structures.  

\subsection{Deformations of holomorphic Poisson structures}
As we saw in Section~\ref{gaugepo}, the pair $(I,\sigma)$ defines, and is  determined by, the complex Dirac structure $L_\sigma$.  
%For the results in this section it will be convenient to change the conventional sign used in the definition of the Dirac structure. That is, we will use instead 
%\begin{equation}\label{lsig}
%L_\sigma = \{X +\sigma\xi + \xi\ :\ X\in T_{0,1}, \xi\in T^*_{1,0}\}.
%\end{equation}
A natural Dirac structure transverse to this one and independent of the Poisson structure is 
\begin{equation}\label{ndir}
N = T_{1,0}\oplus T^{*}_{0,1}.
\end{equation}
The Courant bracket on sections of $N$, given by $[X+\xi,Y+\eta] = [X,Y] + i_{X}\del \eta - i_{Y}\del \xi$, extends to a graded Lie bracket (also denoted by $[\cdot,\cdot]$) on sections 
\(C^{\infty}(M,\wedge^{\bullet} N)=\oplus_{p,q}\Omega^{0,q}(\wedge^{p}T_{1,0}) 
\).
The canonical inner product on $T\oplus T^{*}$ identifies $N$ with $L_{\sigma}^{*}$, so that the de Rham complex associated to the Lie algebroid $L_{\sigma}$ provides the above graded Lie algebra with a differential, namely $d_{\sigma} = \delbar + \del_{\sigma}$, where $\del_{\sigma} \psi = [\sigma,\psi]$, making the following a differential graded Lie algebra: 
\begin{equation}\label{dgla}
\left(\bigoplus_{p,q}\Omega^{0,q}(\wedge^{p}T_{1,0}),\ d_{\sigma}=\delbar + \del_{\sigma},\ [\cdot,\cdot ]\right).
\end{equation}
Obtaining a differential graded Lie algebra from a pair of transverse Dirac structures is a standard result of~\cite{MR1472888}, where it is also explained that Maurer-Cartan elements describe deformations of the original Dirac structure in question.  In our case, an element 
\begin{equation}\label{decompeps}
\eps=\rho + \phi + \gamma \in \Omega^{0,0}(\wedge^{2}T_{1,0}) \oplus \Omega^{0,1}(T_{1,0}) \oplus \Omega^{0,2}
\end{equation}
that satisfies the Maurer-Cartan equation $d_{\sigma}\eps + \tfrac{1}{2}[\eps,\eps]=0$ describes the following deformation of $L_{\sigma}$:
\begin{equation}\label{epslsigma}
L_{\sigma}^{\eps}= \{X + \phi X + \gamma X +  \sigma\zeta + \rho\zeta + \zeta - \phi^{*}\zeta  \ :\ X\in T_{0,1},\ \zeta\in T^{*}_{1,0}\}.
\end{equation}
Such a general deformation will, however, not necessarily define a holomorphic Poisson structure; instead, it describes (for $\eps$ sufficiently small) the complex Dirac structure underlying a generalized complex structure.  For the deformation to remain holomorphic Poisson, it is necessary and sufficient that $\gamma$, the component of the deformation lying in $\Omega^{0,2}$, vanish. Indeed (see \cite[Appendix A]{MR2065506}), deformations of holomorphic Poisson structure are controlled by the $p\geq 1$ truncation of the differential graded Lie algebra~\eqref{dgla}.
\begin{proposition}[{\cite[Appendix A]{MR2065506}}]
Deformations of the pair $(I,\sigma)$ are controlled by the $p\geq 1$ truncation of the dgLa~\eqref{dgla}, i.e. $(\Omega^{0,\bullet}(\wedge^{\geq 1}T_{1,0}), d_\sigma, [\cdot,\cdot])$.
\end{proposition}
Therefore, deformations of $(I,\sigma)$ are given by sections $\eps = \rho + \phi$ as in~\eqref{decompeps}, satisfying the Maurer-Cartan equation, which decomposes into three parts:
\begin{align}
    \delbar\phi + \tfrac{1}{2}[\phi,\phi] = 0\label{defcx}\\
    \delbar\rho + \del_\sigma \phi + [\rho,\phi] = 0\label{newhol}\\
        \del_\sigma \rho + \tfrac{1}{2}[\rho,\rho] = 0\label{jaco}
\end{align}

To first order, therefore, a deformation of holomorphic Poisson structure is given by a solution to the linearized system
\begin{equation}\label{firstor}
    \delbar\phi = 0\qquad \delbar\rho + \del_\sigma\phi = 0\qquad \del_\sigma \rho = 0.
\end{equation}
Identifying infinitesimal deformations when they differ by an infinitesimal diffeomorphism, we obtain an expression for the tangent space to the moduli space $\MM$ of deformations of holomorphic Poisson structure: 
\begin{equation}
T_{[(I,\sigma)]}\MM = \frac{\{\rho+\phi\ :\ d_\sigma(\rho+\phi) = 0\}}{\{d_\sigma Y\ :\ Y\in \Omega^{0,0}(T_{1,0})\}} 
= \HH^1(\X^{\geq 1}[1]),
\end{equation}
where the last expression uses the notation of hypercohomology for a complex of sheaves:  $\X^{\geq 1}[1]$ denotes the truncated complex of holomorphic multivector fields:
\begin{equation}
{\X}^{\geq 1}[1] =
\xymatrix{ \X^1\ar[r]^-{\del_{\sigma}}& {\X}^{2}\ar[r] &\cdots},\qquad \X^k = \OO(\wedge^k T_{1,0}).
\end{equation}

\subsection{Hitchin's unobstructedness result} 
We now review and generalize some results of~\cite{MR3024823}, relaxing certain hypotheses as suggested by~\cite{MR3007085}. We focus on Hitchin's unobstructedness result, which states that on a compact holomorphic Poisson manifold, Kodaira-Spencer classes in the image of the map $[\sigma]:H^{1}(M,\Omega^{1})\to H^{1}(M,T)$ induced by the Poisson tensor $\sigma$ are unobstructed, assuming that the natural map $H^{2}(M,\CC)\to H^{2}(M,\OO)$ is surjective. 
We observe that the mechanism behind the proofs given in~\cite{MR3024823,MR3007085} depends on two basic facts. 

First, the projection $\pi:L_\sigma\to\Tc$, the anchor of the Lie algebroid $L_\sigma$, has dual $\pi^*:\cT\to L^*_\sigma\cong N$ given by 
\begin{equation}
\pi^* = -\sigma\oplus\id{}:T^*_{1,0}\oplus T^*_{0,1}\to T_{1,0}\oplus T^*_{0,1},
\end{equation}
inducing the following dgLa homomorphism:
\begin{equation}\label{dglahom}
\left(\oplus_{p,q}\Omega^{p,q}(M),\ d=\delbar + \del,\ [\cdot,\cdot ]_{\sigma}\right)\xrightarrow{\oplus_{p}\wedge^{p}\pi^*}
\left(\oplus_{p,q}\Omega^{0,q}(\wedge^{p}T_{1,0}),\ d_{\sigma}=\delbar + \del_{\sigma},\ [\cdot,\cdot ]\right).
\end{equation}
The domain of this morphism is the dgLa defined by the complexified de Rham complex of $M$,  equipped with the Koszul bracket $[\cdot,\cdot]_{\sigma}$, given on 1-forms $\xi,\eta\in \Omega^{1}_{\CC}$ by $[\xi,\eta]_{\sigma} = -(L_{\sigma\xi}\eta - L_{\sigma\eta}\xi - d\sigma(\xi,\eta))$.
This fact, first shown in~\cite[\S 6.6]{MR1077465}, may be seen as follows.  That $\pi^*$ commutes with differentials follows from dualizing the Lie algebroid morphism $\pi:\L_\sigma\to \Tc$. On the other hand, the fact that $\pi^*$ is bracket-preserving follows from the fact that $\cT$ may be identified with the Dirac structure $\Gamma_{-\sigma}=\{-\sigma(\eta) + \eta\ :\ \eta\in\cT\}$, from which it inherits the bracket $[\cdot,\cdot]_\sigma$. Since $\Gamma_{-\sigma} = N - L_{\sigma}$, it has a natural Lie algebroid morphism to $N$ given by $-\sigma(\zeta) + \zeta \mapsto -\sigma(\zeta) + \zeta^{0,1}$, which coincides with $\pi^*$ after the identification $\cT \cong L_{-\sigma}$. 
%
%
% for $N$ given by~\eqref{ndir}, coincides with the graph of $\sigma:\cT\to\Tc$, namely $\Gamma_{\sigma} = \{\sigma(\eta) + \eta\ :\ \eta\in\cT\}$.  This implies that we have the following fiber product of Lie algebroids, where $\pi,\tau$ denote the projections of $L_\sigma ,N$, respectively, to $\Tc$ and we use the canonical inner product on $\TT M$ to identify $L_\sigma\cong N^*$ and $\Gamma_\sigma^* = T_\CC$: 
%\begin{equation}
%\begin{aligned}\label{sddiag}
%\xymatrix{
%\Gamma_\sigma \ar[r]^{\tau^*}\ar[d]_{-\pi^*} & L_\sigma\ar[d]^\pi\\
%N\ar[r]_\tau & \Tc
%}
%\end{aligned}
%\end{equation}
%Since both $\pi$ and $-\pi^*$ are Lie algebroid morphisms, we obtain, for example, that $-\pi^*$ induces the dgLa morphism~\eqref{dglahom}.  

The significance of this first fact is that we may produce a Poisson deformation by first solving the Maurer-Cartan equation for $\omega\in\Omega^{(2,0) + (1,1)}_\CC$, namely
\begin{equation}\label{mckosz}
     d\omega + \tfrac{1}{2} [\omega,\omega]_\sigma = 0,
\end{equation}
and then transport it to the Poisson deformation complex via $\pi^*$, obtaining the Poisson deformation  
\begin{equation}\label{trnspt}
\rho = \wedge^2\sigma(\omega^{2,0}),\qquad \phi = -\sigma(\omega^{1,1}).
\end{equation}

The second crucial fact concerns solving the Maurer-Cartan equation~\eqref{mckosz}.  As observed by Koszul~\cite{MR837203}, the Koszul bracket $[\cdot,\cdot]_\sigma$ satisfies a Bogomolov-Tian-Todorov lemma, meaning that it can be computed in terms of the operator $\delta_{\sigma} = i_{\sigma} d - d i_{\sigma}$ via the  derived bracket formula:
%    \begin{equation}\label{PBV}
%        e([\alpha,\beta]_\sigma) = [[e(\alpha),\BV],e(\beta)],
%    \end{equation}
%    where $e(\alpha)\rho = \alpha\wedge\rho$ is the exterior product.  
\begin{equation}\label{BTTL}
    [\alpha,\beta]_\pi = \alpha\wedge\BV(\beta) - (-1)^k\BV(\alpha\wedge\beta) + (-1)^{k}\BV(\alpha)\wedge \beta,
\end{equation}
where $\alpha$ has degree $k$ and $\beta$ is of arbitrary degree.  

The strategy employed in~\cite{MR3024823}, then, is to use the Tian-Todorov lemma~\eqref{BTTL} to construct a solution to~\eqref{mckosz} with prescribed first-order part ($\omega(0) = 0$ and $\dot\omega(0)=\omega_1$, an arbitrary closed form with vanishing $(0,2)$ part) and then transport it via~\eqref{trnspt} to a Poisson deformation with underlying Kodaira-Spencer class $[\sigma(\omega_{1}^{1,1})]$, as required.  

We now provide an alternative proof of this result which will be useful for our application to generalized K\"ahler structures.  
We may motivate our approach as follows.  
In~\cite{MR3007085}, the derived bracket formula~\eqref{BTTL} is used to prove a stronger result: the dgLa defined by the Koszul bracket on the de Rham complex is in fact \emph{formal}, that is, it receives a $L_{\infty}$ quasi-isomorphism $\psi= (\psi_1,\psi_2,
\ldots)$ from the de Rham complex equipped with the zero bracket. 
This immediately implies that a solution to the linear Maurer-Cartan equation $d\beta = 0$ is taken by $\psi$ to a solution 
\begin{equation}\label{formc}
    \omega = \psi(\beta)= \psi_1(\beta) + \tfrac{1}{2} \psi_2(\beta,\beta) + \tfrac{1}{3!}\psi_3(\beta,\beta,\beta) + \cdots
\end{equation}
to the nonlinear Maurer-Cartan equation~\eqref{mckosz}.  
In \cite[Theorem 4.2]{Gualtieri:2017aa} it is shown that the formality map~\eqref{formc} given in~\cite{MR3007085} has the following geometric description: so long as $(1+\sigma\beta)$ is invertible, the graphs $\Gamma_{\beta}, \Gamma_{-\sigma}$ in $\Tc\oplus \cT$ are complementary Dirac structures. Therefore $\Gamma_{\beta}$ may be described as the graph of a map $\omega:\Tc\to\Gamma_{-\sigma}$.  Identifying $\Gamma_{-\sigma}\cong \cT$, we see that $\omega$ defines a complex 2-form, and this is precisely $\psi(\beta)$. That is, $\omega$ is the unique complex 2-form such that 
\begin{equation}\label{betaomega}
\Gamma_{\beta} = \{ X - \sigma(\omega X) + \omega X\ :\ X\in \Tc\}.
\end{equation}
The Maurer-Cartan equation for $\omega$ is then precisely the integrability of the Dirac structure~\eqref{betaomega}, which of course holds if and only if $d\beta = 0$.  
We may use the above expression to solve for $\omega$; we obtain 
\begin{equation}
\psi(\beta) = \omega = (1+\beta\sigma)^{-1}\beta = \beta - \beta\sigma \beta + \beta\sigma \beta\sigma \beta + \cdots.
\end{equation}
Comparing this with the formality map~\eqref{formc}, we see that $\tfrac{1}{k!}\psi_k(\beta,\cdots,\beta) = (-\beta\sigma)^{k}\beta$.  We summarize the formality discussion as follows:  
\begin{lemma}\label{siform}
Let $\beta\in\Omega^2_\CC$ such that $d\beta=0$. Then, so long as $1+\beta\sigma:\Tc \to \Tc$ is invertible,
the 2-form given by  
\begin{equation}\label{siex}
    \omega = \psi(\beta)  = (1+\beta\sigma)^{-1}\beta = \beta - \beta\sigma \beta + \beta\sigma \beta\sigma \beta + \cdots
\end{equation}
is a solution to the Maurer-Cartan equation
\begin{equation}\label{mckos}
d\omega + \tfrac{1}{2}[\omega,\omega]_\sigma = 0.
\end{equation}
\end{lemma}
%\begin{proof}
%The Maurer-Cartan equation~\eqref{mckos} describes deformations of $T$ as a Dirac structure, where $\omega$ is viewed as a map $T\to \Gamma_\sigma$, defining the deformation 
%\begin{equation}\label{omd}
%T^\omega = \{X + \omega X + \sigma\omega X\ :\ X\in T\}.
%\end{equation}
%On the other hand, deformations of $T$ as a Dirac structure may also be described using $B$-field gauge transformations, via 
%\begin{equation}\label{bdef}
%    e^B(T) = \{X + \beta X\ :\ X\in T\}.
%\end{equation}
%Setting the two deformations~\eqref{omd}, \eqref{bdef} equal, we obtain $\omega = \beta(1 + \sigma\omega)$, or, dualizing, $\omega = (1 + \omega\sigma)\beta$.  Solving for $\omega$, we obtain the result.
%\end{proof}

This means that we can supplement the above strategy, first solving $d\beta=0$, mapping it to a Maurer-Cartan element $\omega = \psi(\beta)$ and finally to a deformation $\pi^*(\psi(\beta))$ of $L_\sigma$, composing the following morphisms.
\begin{equation}\label{compol}
\xymatrix{(\Omega^\bullet, d, 0)\ar[r]^-{\psi} & (\Omega^\bullet, d, [\cdot,\cdot]_\sigma)\ar[r]^-{\pi^* } & (\X^\bullet, d_\sigma, [\cdot,\cdot])  }
\end{equation}

Given the ideas in Section~\ref{gaugepo}, it is no surprise that a closed 2-form $\beta$ may be used to deform a Poisson structure; indeed, we now show that deforming $L_\sigma$ by $\pi^*(\psi(\beta))$ is equivalent to applying the gauge transformation $\beta$:
\begin{theorem}\label{composb}
Let $\beta\in\Omega^2_\CC$ such that $d\beta=0$, and suppose that $1+\beta\sigma$ is invertible.  Then, the Maurer-Cartan element $\pi^*(\psi(\beta))$ deforms $L_\sigma$ to $e^\beta L_\sigma$.  That is, on Maurer-Cartan elements, the composition~\eqref{compol} defines a gauge transformation of the Dirac structure $L_\sigma$. 
\end{theorem}
\begin{proof}
We must show that the deformation of $L_\sigma$ given by the Maurer-Cartan element $\pi^*(\omega)$ describes precisely the deformation $e^\beta L_\sigma$.  First note that $e^\beta L_\sigma = L_\sigma + \Gamma_\beta$. Then, using~\eqref{betaomega}, we have that 
\begin{equation}\label{betao}
e^\beta L_\sigma = \{\ell + \xi\ :\ \xi = \omega(\pi\ell + \sigma\xi)\}.
\end{equation}
We wish to show that this Dirac structure is the graph of the Maurer-Cartan element $\pi^*(\omega)$; as a map this element is the composition $\pi^*\omega\pi:L_\sigma\to N$.   Splitting $\ell + \xi$ according to the decomposition $L_{\sigma}\oplus N$, we obtain
\begin{equation}
\ell + \xi  = (\ell +\sigma\xi^{1,0} + \xi^{1,0}) + (-\sigma\xi^{1,0} + \xi^{0,1}),
\end{equation}
and using~\eqref{betao}, we verify that this lies in the graph of $\pi^*\omega\pi$:
\begin{equation}
\pi^*\omega\pi(\ell +\sigma\xi^{1,0} + \xi^{1,0}) = -\sigma\xi^{1,0} + \xi^{0,1},
\end{equation}
as required to show $e^\beta L_\sigma$ is the graph of the Maurer-Cartan element $\pi^*(\omega)$. 
\end{proof}

The main problem with the idea of composing the morphisms in~\eqref{compol} to obtain Poisson deformations is that while $\pi^*$ preserves the Hodge filtration, $\psi$ does not.  That is, $\psi(\beta)^{0,2}$ may not vanish even if $\beta^{0,2}=0$.  In order to guarantee that we may always modify $\beta$ so as to kill the $(0,2)$ component of $\psi(\beta)$, we assume that the natural map $H^{2}(M,\CC)\to H^{2}(M,\OO)$ is surjective. We now show that this is enough to produce Maurer-Cartan elements in the truncated Koszul dgLa. 

\begin{theorem}\label{defex}
Let $(M,I,\sigma)$ be a a compact holomorphic Poisson manifold for which the natural map $H^{2}(M,\CC)\to H^{2}(M,\OO)$ is surjective, and fix $\omega_1\in \Omega^{2}(M,\CC)$ such that 
\begin{equation}
\omega_1^{0,2} = 0\qquad\text{and}\qquad d\omega_1 = 0.
\end{equation}
Then there is a family $\beta(t)$ of closed complex 2-forms, holomorphic in $t$, for $t$ in a neighbourhood of $0\in\CC$, with $\beta(0)=0,$ $\del_t\beta(0) = \omega_1$, and with the property that $\omega(t) = \psi(\beta(t)) = (1+\beta(t)\sigma)^{-1}\beta(t)$ satisfies, for each $t$, the conditions 
\begin{equation}
\omega^{0,2} = 0\qquad\text{and}\qquad d\omega  + \tfrac{1}{2}[\omega,\omega]_\sigma = 0.
\end{equation}
Therefore we obtain a family of Maurer-Cartan elements
\begin{equation}\label{thepoisdef}
\eps(t) = \rho(t)+\phi(t)= \wedge^2\sigma(\omega^{2,0}(t)) - \sigma(\omega^{1,1}(t)),
\end{equation} 
defining a deformation of the holomorphic Poisson structure. 
\end{theorem}
\begin{proof}
We construct a time-dependent family $\beta = t\beta_1 + t^2 \beta_2 + \cdots$ of closed 2-forms such that $\psi(\beta)^{0,2}=0$.  The fact that $\pi^*(\psi(\beta))$ is a Maurer-Cartan element has already been explained: \eqref{dglahom} is a dgLa homomorphism. We use $\beta_{\leq k}$ to denote the truncation $t\beta_1 + \cdots + t^k \beta_k$ of $\beta$ to order $k$.

We begin by setting $\beta_1 = \omega_1$, so that $\psi(\beta)^{0,2} = (t \omega_1)^{0,2} = 0 \pmod{t^2}$.
Now assume by induction that $\beta_1,\ldots, \beta_k$ are closed 2-forms chosen such that $\psi(\beta)^{0,2} = 0 \pmod{t^{k+1}}$. We show how $\beta_{k+1}$ can be chosen so that $\psi(\beta)^{0,2}$ vanishes to the next order. 

Since $\psi(\beta)^{0,2} = 0 \pmod{t^{k+1}}$, we have that 
\begin{equation}
\psi(\beta_{\leq k})^{0,2} = t^{k+1} r_{k+1} \pmod{t^{k+2}},
\end{equation}
for $r_{k+1}\in\Omega^{0,2}$. But since $\beta_{\leq k}$ is closed, $\psi(\beta_{\leq k})$ satisfies the Maurer-Cartan equation, which implies that $\delbar r_{k+1} = 0$, defining a class in $H^2(M,\OO)$.
Finally, we have 
\begin{equation}
\psi(\beta)^{0,2} = t^{k+1}(r_{k+1} + \beta_{k+1}^{0,2}) \pmod{t^{k+2}},
\end{equation}
so that the surjectivity of $H^{2}(M,\CC)\to H^{2}(M,\OO)$ allows the choice of a closed form $\beta_{k+1}$ such that $\psi(\beta)^{0,2} = 0 \pmod{t^{k+2}}$. We make this choice as follows: choose a Hermitian metric as well as a splitting $\Hh\subset H^2(M,\CC)$ for the projection to $H^2(M,\OO)$, and let $\wt r_{k+1}\in \Omega^2(M,\CC)$ be the unique harmonic lift of $r_{k+1}$ such that $[\wt r_{k+1}]\in\Hh$.  Then 
\begin{equation}
\wt r_{k+1}^{0,2} - r_{k+1} = \delbar \gamma_{k+1}
\end{equation} 
for a unique $\gamma_{k+1}\in\Omega^{0,1}$ in the image of $\overline{\partial}^*$.  Then define $\beta_{k+1} = -\wt r_{k+1} + d\gamma_{k+1}$.  
Finally, standard elliptic estimates may be used to conclude, as in~\cite{MR2214741, MR2669364}, that the series $\beta(t)=\sum_kt^k\beta_k$ converges.
\end{proof}

Differentiating~\eqref{thepoisdef} at $t=0$, we obtain a map of cohomology groups for the $p\geq 1$ truncations of the complexes appearing in~\eqref{dglahom}.  In terms of hypercohomology, we have the homomorphism
\begin{equation}\label{unob}
\xymatrix@R=1em{\HH^1(\Omega^{\geq 1}[1])\ar[r]^{\sigma}&  \HH^1(\X^{\geq 1}[1]) \\
[\omega_1]\ar@{|->}[r] 
\ar@<0.7ex>@{}[u]|-*=0[@]{\in}
& [\wedge^2\sigma(\omega_1^{2,0}) - \sigma(\omega_1^{1,1})]\ar@<0.7ex>@{}[u]|-*=0[@]{\in}
}
\end{equation}
In the above diagram, ${\Omega}^{\geq 1}[1]$ denotes the truncated holomorphic de Rham complex:
\begin{equation}
{\Omega}^{\geq 1}[1] =
\xymatrix{ {\Omega}^{1}\ar[r]^-{d}& {\Omega}^{2}\ar[r] &\cdots},
\qquad \Omega^k=\OO(\wedge^{k} T^*_{1,0}).
%\qquad 
%{\X}^{\geq 1}[1] =
%\xymatrix{ \X^1\ar[r]^-{d_{\sigma}}& {\X}^{2}\ar[r] &\cdots}
\end{equation}
Projecting each complex to its first position, we obtain the natural induced map
\begin{equation}
\begin{aligned}\label{unobpr}
\xymatrix@R=1em{H^1(\Omega^{1})\ar[r]^{\sigma}&  H^1(T) \\
[\omega_1^{1,1}]\ar@{|->}[r] 
\ar@<0.7ex>@{}[u]|-*=0[@]{\in}
& [-\sigma(\omega_1^{1,1})]\ar@<0.7ex>@{}[u]|-*=0[@]{\in}
}
\end{aligned}
\end{equation}
the unobstructedness of whose image was the focus of~\cite{MR3024823}:
\begin{corollary}
Any first-order deformation of the compact holomorphic Poisson manifold $(M,I,\sigma)$ in the image of the map~\eqref{unob} is unobstructed.  In particular, any first-order deformation of the underlying complex structure in the image of~\eqref{unobpr} is unobstructed.
\end{corollary}

The proof of Theorem~\ref{defex} gives some information about the behaviour of the cohomology class $[\beta(t)]$ in the construction of the deformation.  As done in the proof, if we choose a splitting $\Hh\subset H^2(M,\CC)$ of the map $\pi:H^2(M,\CC)\to H^2(M,\OO)$, then $\beta(t)$ can be chosen so that $[\beta(t) - t\omega_1]\in \Hh$. In particular, if $H^2(M,\OO)=\{0\}$, we obtain the following generalization of~\cite[Proposition 7]{MR3024823} (where the result was obtained for a generically symplectic Poisson structure).
\begin{corollary}
If $H^2(M,\OO)=\{0\}$, 
 then in Theorem~\ref{defex}, we may choose $\beta(t)$ such that $[\beta(t)] = t[\omega_1]$ in $H^2(M,\CC)$, producing a linear dependence on $t$ for the periods of all of the holomorphic symplectic leaves of the resulting family of holomorphic Poisson structures.
\end{corollary}
We now give an example where the periods of the deformation do not vary linearly but instead quadratically.  
\begin{example}[Twistor family of a hyperkahler manifold]
Let $(M, g, I_1, I_2, I_3)$ be a hyperkahler manifold, with triple of K\"ahler forms $(\omega_1,\omega_2,\omega_3)$.  Consider the holomorphic Poisson structure $(I_1, \sigma_1)$ inverse to the holomorphic symplectic form $\Omega_1 = \omega_2 + i \omega_3$:
\begin{equation}
\sigma_1 = \tfrac{1}{4}(\omega_2^{-1} - i \omega_3^{-1}).
\end{equation}
We may then apply Theorem~\ref{defex} to the $(1,1)$ form $2i\omega_1$, obtaining an unobstructed deformation in the direction $2i\sigma_1(\omega_1)$.  In the construction of the closed form $\beta = t\beta_1 + t^2\beta_2 + \cdots$ something very special occurs: the first remainder term is actually closed:
\begin{equation}
r_2 = (2i)^2\omega_1 \sigma_1\omega_1 = -\omega_2 + i\omega_3 = -\ol{\Omega_1}.
\end{equation}
As a result, we may simply take $\beta_2 = \ol{\Omega_1}$.  Then because $\sigma_1\omega_1\sigma_1=0$ and $\sigma_1\beta_2 = \beta_2\sigma_1=0$, there are no higher terms in $\psi(\beta)^{0,2}$ to cancel. So we have a solution which terminates at the quadratic term:
\begin{equation}
\beta(t) = 2it\omega_1 + t^2\ol{\Omega_1}.
\end{equation}
This family is very familiar: we may describe it in terms of the inverse family of holomorphic symplectic structures:
\begin{equation}
\Omega(t) = \Omega_1 + 2i t \omega_1 + t^2\ol\Omega_1,
\end{equation}
which is the standard twistor family associated to the hyperkahler manifold.
\end{example}

We shall study the properties of the Poisson deformation given by Theorem~\ref{defex} in the next section, and then apply it to the construction of generalized K\"ahler structures in Section~\ref{applic}.  To conclude this section, we use the techniques of~\cite{Gualtieri:2017aa} to describe explicitly the new holomorphic Poisson structure obtained by applying the deformation $\eps(t) = \rho(t)+\phi(t)$ given in~\eqref{thepoisdef}.

Recall that the dgLa~\eqref{dgla} controlling Poisson deformations derives from the pair of transverse Dirac structures $(L_\sigma, N)$ defined in~\eqref{lsig}, \eqref{ndir}.  The Maurer-Cartan element $\eps = \rho + \phi$ deforms $L_\sigma$ to 
\begin{equation}
L^\eps_\sigma =\{X + \phi X  +  \sigma\zeta + \rho\zeta + \zeta - \phi^{*}\zeta  \ :\ X\in T_{0,1},\ \zeta\in T^{*}_{1,0}\}.
\end{equation}

Since $L_{\sigma}^{\eps}$ is still transverse to $N$, it induces a modification of the differential in~\eqref{dgla}, namely 
\begin{equation}
d_{\sigma}^{\eps} = (\delbar + [\phi,\cdot ]) + (\del_{\sigma} + [\rho,\cdot]),
\end{equation}
yielding the usual interpretation of the Maurer-Cartan equation  as the condition $(d^{\eps}_{\sigma})^{2}= 0$ for a new differential graded Lie algebra:
\begin{equation}\label{dgladef}
\left(\bigoplus_{p,q}\Omega^{0,q}(\wedge^{p}T_{1,0}),\ d^{\eps}_{\sigma}, [\cdot,\cdot]\right).
\end{equation}

The first equation~\eqref{defcx} is familiar from the deformation theory of complex structures, and holds if and only if the complex distribution
\begin{equation}\label{newcx}
T^{\phi}_{0,1} = \{X + \phi X\ :\ X\in T_{0,1}\}
\end{equation}
is involutive.
%where we view $\phi\in \Omega^{0,1}(T_{1,0})$ as a bundle map $T_{0,1}\to T_{1,0}$.  
As long as $(1-\phi\ol\phi)$ is invertible, this defines the $-i$-eigenspace of a new complex structure (and we use $T^{\phi}_{1,0}$ to denote its $+i$-eigenspace).  Rewriting $L^{\eps}_{\sigma}$ in terms of the new complex structure, we obtain 
\begin{equation}
L_{\sigma}^{\eps}= \{X +   \wt\sigma\xi + \xi  \ :\ X\in T^{\phi}_{0,1},\ \xi\in (T^{\phi}_{1,0})^{*}\},
\end{equation}
where the deformed Poisson tensor $\wt \sigma$ is given by 
\begin{equation}\label{newpoi}
\wt \sigma = \wedge^{2}P^{\phi}_{1,0} (\sigma+\rho),
\end{equation}
where $P^{\phi}_{1,0}$ is the projection operator on $T_{\CC}$ with kernel $T^{\phi}_{0,1}$ and image $T^{\phi}_{1,0}$.  This gives an explicit description of the deformed Poisson tensor as the $(2,0)$ component of $\sigma +\rho$ relative to the new complex structure. 

Just as $N$ is transverse to $L_{\sigma}$, so too we have that $N_{\phi}$, defined below, is transverse to $L^{\eps}_{\sigma}$:
\begin{equation}
N_{\phi} = T^{\phi}_{1,0}\oplus (T^{\phi}_{0,1})^{*}.
\end{equation}
Therefore, using the new complex structure and Poisson bivector, we obtain, just as in~\eqref{dgla}, a differential graded Lie algebra structure on sections of $\wedge^{\bullet} N_{\phi}$: 
\begin{equation}\label{dgla2}
\left(\bigoplus_{p,q}\Omega^{0,q}_{\phi}(\wedge^{p}T_{1,0}^{\phi}),\ d_{\wt\sigma}=\delbar_{\phi} + \del_{\wt\sigma},\ [\cdot,\cdot ]\right),
\end{equation}
where $\delbar_{\phi}$ is the Dolbeault operator in the new complex structure.

There are two key observations about the relationship between the dgLas~\eqref{dgladef} and~\eqref{dgla2}. The first is that, since both $N$ and $N_{\phi}$ are transverse to (hence dual to) $L^{\eps}_{\sigma}$, they are naturally isomorphic via the map $\Psi:T_{1,0}\oplus T^{*}_{0,1}\to T^{\phi}_{1,0}\oplus (T^{\phi}_{0,1})^{*}$, given by 
\begin{equation}\label{upt}
\Psi = \begin{pmatrix}
P^{\phi}_{1,0} & P_{1,0}^{\phi}(\sigma + \rho) (P_{0,1}^{\phi})^{*}\\
0 & (P_{0,1}^{\phi})^{*}
\end{pmatrix}
\end{equation}
As a result, $\Psi$ induces an isomorphism of cochain complexes from~\eqref{dgladef} to~\eqref{dgla2}.  The second is that this isomorphism need not be bracket-preserving; indeed, in~\cite{Gualtieri:2017aa} it is explained that $\Psi$ is the first in a sequence of maps defining an $L_\infty$ isomorphism between the two dgLas.  

Furthermore, while $\Psi$ does not respect the bigrading of each dgLa, it does, due to the upper-triangular expression~\eqref{upt}, respect the decreasing filtration in the $p$ index analogous to the Hodge filtration.  
%Because of this, we observe that Equation~\ref{newhol}, which states that $(\delbar + [\phi,\cdot])(\sigma+\rho) = 0$, implies that $\wt\sigma=\wedge^{2}P_{1,0}^{\phi}(\sigma+\rho)$ is holomorphic in the new complex structure, that is, $\delbar_{\phi}\wt\sigma = 0$.  Similarly, Equation~\ref{jaco}, which states that $(\del_{\sigma} +[\rho,\cdot])(\sigma + \rho) = 0$, implies that the Jacobi identity holds for $\wt \sigma$.   As a result we have an explicit expression for the deformed holomorphic Poisson bivector, namely 
%\begin{equation}
%\wt\sigma = \wedge^{2}P_{1,0}^{\phi}(\sigma+\rho).
%\end{equation}
%The isomorphism $\Psi$ 
This allows us to reason about the geometry of the  deformed holomorphic Poisson structure using the modified differential on the original Dolbeault complex~\eqref{dgladef}.  For example, a function $f$ is holomorphic in the new complex structure, i.e. $\delbar_{\phi}f=0$, if and only if $\delbar f + [\phi,f] = 0$.  Also,
if $Z\in \Omega^{0,0}(T_{1,0})$ satisfies $\delbar Z + [\phi,Z]=0$, then $Z_{\phi}=P^{\phi}_{1,0}(Z)$ is holomorphic in the new complex structure, and if, in addition, $[\sigma + \rho,Z]=0$, so that $d^{\eps}_{\sigma}Z = 0$,  then $Z_{\phi}$ is a Poisson vector field for $\wt\sigma$.

%For convenience, we define $P_{1,0}^{\phi}$ to be the projection operator to the $+i$-eigenspace $T^{\phi}_{1,0}$ with kernel $T^{\phi}_{0,1}$, given explicitly by 
%\[
%P^{\phi}_{1,0}(Z) = (1 + \ol\phi)(1-\phi\ol\phi)^{-1}(Z_{1,0} - \phi Z_{0,1}),
%\] 
%where $Z_{1,0},Z_{0,1}$ are the components of $Z\in T\otimes\CC$ in the original $\pm i$-eigenbundles $T_{1,0}, T_{0,1}$, respectively.  

\section{Hamiltonian families and Morita equivalence}\label{sevfam}

In this section, we explain that the deformations of holomorphic Poisson structure constructed in the previous section may be characterized as \emph{Hamiltonian families}, introduced by \v{S}evera~\cite{Severa:aa,MR1978549} in the real Poisson category.  We also show how this implies that the deformations of Poisson structure remain in the same Morita equivalence class. As this section is not needed for our main results in Section~\ref{applic}, we will omit details, which we hope will appear elsewhere.

\subsection{Real Hamiltonian families and Dirac structures}\label{realdirac}
Let $\pi_0$ be a real Poisson structure on $M$, and let $B_t\in\Omega^2(M,\RR), t\in\RR$ be a family of real closed 2-forms with $B(0)=0$. Applying a gauge transformation to $\pi_0$ gives rise to the following family of Poisson structures for $t$ sufficiently small in $\RR$:
\begin{equation}
\pi_t = \pi_0(1 + B_t\pi_0)^{-1}.
\end{equation}
This family satisfies the differential equation
\begin{equation}\label{hamr}
\dot \pi_t = -\pi_t \dot B_t \pi_t = -\wedge^{\!2}\!\pi_t (\dot B_t),
\end{equation}
expressing the fact that $\pi_t$ has velocity given by the infinitesimal gauge transformation by $\dot B_t$.  This is a special case of what \v{S}evera called a \emph{Hamiltonian} family of Poisson structures~\cite{Severa:aa,MR1978549}.  Notice the analogy: if $A$ is a closed 1-form, then $\pi_t(A)$ is (locally) an infinitesimal Hamiltonian symmetry, whereas for $B$ a closed 2-form $\wedge^2\pi_t(B)$ is a Hamiltonian deformation.  As \v{S}evera remarked, the differential equation~\eqref{hamr} may be expressed as the integrability condition of an interesting Dirac structure $D\subset (T\oplus T^*)(M\times \RR)$ on the total space of the family, which we may describe as follows:
\begin{equation}
D|_{M\times\{t\}} = \{ f\del_t +  \pi_t\xi + \xi\ :\ f\in\Omega^0(M), \xi\in \Omega^1(M)\}
\end{equation}
\begin{proposition}\label{realhamasdir}
The Dirac structure $D$ defined above is involutive with respect to the 3-form $H = dt\wedge \dot B_t\in\Omega^3(M\times\RR)$ if and only if $\pi_t$ is a family of Poisson structures satisfying the differential equation~\eqref{hamr} for a closed 2-form $B$.  
\end{proposition}
\begin{proof}
Involutivity for sections of the form $\pi_t\xi + \xi$ holds since $\pi_t$ is a Poisson structure. The only remaining condition is the Courant bracket 
\begin{equation}
\begin{aligned}
{[}\del_t, \pi_t\xi + \xi{]}_H &= \dot\xi + \pi_t\dot\xi + \dot\pi_t\xi + i_{\pi_t\xi}i_{\del_t} (dt\wedge\dot B_t)\\
&= \dot\xi + \pi_t\dot\xi + \dot\pi_t\xi + i_{\pi_t\xi}\dot B_t,
\end{aligned}
\end{equation}
which lies in $D$ for all $\xi$ if and only if $\pi_t\dot B_t\pi_t= - \dot\pi_t$, as required. 
\end{proof}

The Dirac structure $D$ in the construction above is characterized by the fact that it is complementary to the Dirac structure defined by the projection $p:M\times\RR\to\RR$ defining the family itself, namely 
\begin{equation}\label{famdi}
F = \ker{p_*} \oplus\ann{\ker{p_*}} = TM\oplus T^*\RR.
\end{equation}
Since $D$ and $F$ are complementary, we may take the dirac sum $D-F$, and this recovers the Poisson structure on $M\times\RR$ defined by the family $\pi_t$:
\begin{equation}
D - F = \Gamma_{\pi_t}.
\end{equation}
We also see from this that $D$ fits into the following exact sequence:
\begin{equation}
\xymatrix{0\ar[r] & \Gamma_{\pi_t}\ar[r] & D\ar[r]^-{p_*}& p^* T\RR\ar[r] & 0}
\end{equation}
Such an exact sequence of Lie algebroids leads to a similar short exact sequence of the corresponding Lie groupoids\footnote{We do not address issues of integrability and source-connectedness here.}
\begin{equation}
\xymatrix{0\ar[r] & G_{\pi_t}\ar[r] & G_D\ar[r]^-{q}& \Pi_1(\RR)\ar[r] & 0},
\end{equation}
with the important consequence that if $[\gamma]\in\Pi_1(\RR)$ is a path from $0$ to $t$, then $Q_\gamma = q^{-1}([\gamma])$ defines a Morita equivalence between the Lie groupoids $G_{\pi_0}$ and $G_{\pi_t}$.  Furthermore, as shown in~\cite{MR2068969}, the fact that $D$ is a Dirac structure means that its Lie groupoid $G_D$ is endowed with a distinguished 2-form defining a multiplicative presymplectic structure.  The restriction of this 2-form to $Q_\gamma$ is then a symplectic structure making $Q_\gamma$ into a symplectic Morita bimodule, in the sense of~\cite{MR1104935}, between the Poisson manifolds $(M,\pi_0)$ and $(M,\pi_t)$.    In this way, the Dirac structure $D$ gives an alternative approach to the result, originally observed in~\cite{MR1973074}, that gauge-equivalent Poisson structures are Morita equivalent.

\subsection{Complex Hamiltonian families and Dirac structures}\label{cxhamfam}

The result of Theorems~\ref{composb} and~\ref{defex} was to produce a holomorphic family $\beta = \{\beta_t, t\in\CC\}$ of closed complex 2-forms, with $\beta_0=0$ and $\del_t\beta(0)=\omega_1$, such that its action on the holomorphic Poisson structure $\sigma_0$ by gauge transformation produces a family $\sigma_t$ of holomorphic Poisson structures.  In terms of the corresponding Dirac structures we have $e^{\beta_t}L_{\sigma_0} = L_{\sigma_t}$.  By assumption, $\del_t\beta(0) = \omega_1$ is of type $(2,0)+(1,1)$ relative to the initial complex structure $I_0$.  We may however repeat this analysis at any other time $t_0$:  first note that 
\begin{equation}
L_{\sigma_t} = e^{\beta_t-\beta_{t_0}} e^{\beta_{t_0}}L_{\sigma_0} =e^{\beta_t-\beta_{t_0}}L_{\sigma_{t_0}} = e^{\wt\beta_t}L_{\sigma_{t_0}},
\end{equation}
for $\wt\beta_t = \beta_t - \beta_{t_0}$.
As a result of Theorem~\ref{composb}, this gives rise to the following Maurer-Cartan element for the Koszul bracket of $\sigma_{t_0}$:
\begin{equation}
\wt \omega_t = (1+\wt\beta_t\sigma_{t_0})^{-1}\wt\beta_t.
\end{equation}
Since the resulting deformation of $L_{\sigma_0}$ remains Poisson, $\wt\omega_t$ must remain of type $(2,0) + (1,1)$ relative to the complex structure $I_{t_0}$. As a result, we have that the closed form
\begin{equation}\label{equalvel}
\left.\del_t\wt \omega\right|_{t_0} = \left.\del_t\wt\beta\right|_{t_0} = \left.\del_t\beta\right|_{t_0}
\end{equation} 
%abbreviated as $\dot\beta(t_0)$, 
has type $(2,0) + (1,1)$ relative to the complex structure $I_{t_0}$. In conclusion, we obtain the following ``differential'' characterization of the family of Poisson structures constructed in Theorem~\ref{defex}.  We may consider it to be a holomorphic version of the Hamiltonian family described in the previous section.
%\begin{proposition}\label{hamfamcx}
%The family $(I_t, \sigma_t)$ of holomorphic Poisson structures constructed in Theorem~\ref{defex} via the family of closed 2-forms $\beta$ has the following derivative at each time $t$, expressed as a first-order deformation of $(I_t,\sigma_t)$ solving~\eqref{firstor}:
%\begin{equation}\label{cxham}
%\begin{aligned}
%\phi_t &= \sigma_t (\del_t\beta)^{1,1},\\
%\rho_t &=\wedge^2\sigma_t(\del_t\beta)^{2,0},
%\end{aligned}
%\end{equation}
%where the projections to the $(2,0)$ and $(1,1)$ components are taken with respect to the complex structure $I_t$ at time $t$.
%\end{proposition}
%In particular, we see that at all times $t$, the Poisson deformation has Kodaira-Spencer class in $\HH^1(\X_t^{\geq 1}[1])$ given by the image of $[\del_t\beta] \in \HH^1(\Omega_t^{\geq 1}[1])$ under $\wedge^\bullet\sigma$, generalizing~\cite[Proposition 5]{MR3024823}.
%
%\begin{remark}
%Note that the expression~\eqref{cxham} does not provide the actual time derivatives of the tensors $(I_t,\sigma_t)$, but only the infinitesimal Maurer-Cartan elements describing these derivatives.  For the actual variation of the tensors, we must use the projections as in~\eqref{newpoi}, and this leads to the following system of differential equations on tensors:
%\begin{equation}\label{cxham2}
%\begin{aligned}
%\del_t I_t &= -2i \del_t\phi(t) = -2i\sigma(\del_t \beta)^{1,1}\\
%\del_t\sigma_t &=\sigma_t(\del_t\beta)^{2,0}\sigma_t + 
%\ol\sigma_t(\ol{\del_t\beta)}^{1,1}\sigma_t+
%\sigma_t(\ol{\del_t\beta)}^{1,1}\ol\sigma_t.
%\\
%\end{aligned}
%\end{equation}
%\end{remark}
\begin{theorem}\label{hamfamcx}
The family $(I_t, \sigma_t)$ of holomorphic Poisson structures constructed in Theorem~\ref{defex} via the family of closed 2-forms $\beta$ has the following derivative at each time $t$:
\begin{equation}\label{cxham}
\begin{aligned}
\del_t I_t&= -2i \sigma_t(\alpha^{1,1}),\\
\del_t \sigma_t &=\wedge^2\sigma_t(\alpha^{2,0}),\\
\delbar_t \sigma_t &= \ol\sigma_t\ol{\alpha^{1,1}}\sigma_t+\sigma_t\ol{\alpha^{1,1}}\ol\sigma_t.
\end{aligned}
\end{equation}
where $\alpha=\del_t\beta$ and the projections to the $(2,0)$ and $(1,1)$ components are taken with respect to the complex structure $I_t$ at time $t$.
\end{theorem}
\begin{remark}
In particular, we see from the first two equations that at all times $t$, the Poisson deformation has Kodaira-Spencer class in $\HH^1(\X_t^{\geq 1}[1])$ given by the image of $[\del_t\beta] \in \HH^1(\Omega_t^{\geq 1}[1])$ under $\wedge^\bullet\sigma$, generalizing~\cite[Proposition 5]{MR3024823}.
\end{remark}
\begin{proof}
Let $\phi(t)\in\Omega^{0,1}(T_{1,0}(I_0))$ be a Maurer-Cartan element deforming the complex structure $I_0$ to $I_t$. This means that the real operator on $\Tc = T_{1,0}(I_0)\oplus T_{0,1}(I_0)$ 
\[
A = \begin{pmatrix} 1 & \phi\\ \ol{\phi} & 1  \end{pmatrix}
\]
relates the two holomorphic tangent bundles via $A (T_{1,0}(I_0)) = T_{1,0}(I_t)$.  Therefore $I_t = AI_0 A^{-1}$.  Note also that $A$ is only invertible if $S  = (1 - \phi\ol{\phi})^{-1}$ exists.
The projection operator to $T_{1,0}(I_t)$ is given by 
\[
	P = A\begin{pmatrix}1&0\\0&0\end{pmatrix} A^{-1}.
\]
So, the derivative along the complex vector field $\del_t$ is given by 
\begin{equation}\label{pdot}
\del_t P = [\del_t A, P] = [\del_t\phi(t), P] = - \del_t\phi(t).
\end{equation}
So, using the fact that $P = \tfrac{1}{2}(1-iI_t)$, we obtain that $\del_t I_t = - 2i\del_t\phi(t)$.  In our situation $\phi(t) = \sigma_{t_0}(\wt \omega_t^{1,1})$, so by~\eqref{equalvel}, and setting $t_0=t$, we obtain the first equation in~\eqref{cxham}.  
For the second equation, we use the fact that $\sigma_{t} = \wedge^2 P (\sigma_{t_0} + \wedge^2\sigma_{t_0}(\wt\omega_t^{2,0}))$. Differentiating, we obtain 
\[
\del_t\sigma|_{t_0} = \del_t P|_{t_0} \sigma_{t_0} + \sigma_{t_0} \del_t P^*|_{t_0} + \wedge^2\sigma_{t_0}(\del_t\wt\omega|_{t_0}^{2,0}).
\]
The first two terms vanish, using~\eqref{pdot} and the fact that $\sigma(\wt\omega^{1,1}_t)\sigma = 0$ since $\sigma$ has type $(2,0)$. The result then follows from~\eqref{equalvel}.  For the last equation, we proceed in the same way, now using that $\delbar_t A = \ol{\del_t\phi(t)}$.  This then implies that $\delbar_t P = \ol{\del_t\phi(t)}$, giving 
\[
\delbar_t\sigma|_{t_0} = \delbar_t P|_{t_0} \sigma_{t_0} + \sigma_{t_0} \delbar_t P^*|_{t_0} + \wedge^2\sigma_{t_0}(\delbar_t\wt\omega|_{t_0}^{2,0}).
\]
In this case, the third term vanishes since $\beta(t)$ is analytic in $t$, whereas the first two terms give the required result, since $\delbar_t P = \ol{\del_t\phi(t)} = \ol{\sigma_t(\del_t\beta^{1,1})}$.
\end{proof}

The system of equations~\eqref{cxham} has an interesting interpretation in terms of Dirac geometry, which we now briefly describe.  The underlying deformation of complex structure determines a complex structure on the total space $X = M\times\CC$ with operator $I|_{M\times\{t\}} = I_t\oplus i$, for which the projection $p:X\to\CC$ is a holomorphic submersion.  Since $\del_t\beta$ is closed and has vanishing $(0,2)$ component relative to $I_t$, it defines a closed $(2,0)+(1,1)$ form on the total space $X$, and hence a closed $(3,0)+(2,1)$ form 
\begin{equation}\label{Hdef}
H = dt\wedge\del_t\beta.
\end{equation}
By the classification of holomorphic Courant algebroids in~\cite{MR3232003}, such a form defines a holomorphic Courant algebroid in the following way. First, we use the $(2,1)$ component $H^{2,1}$ to deform the holomorphic structure on the vector bundle $E = T_{1,0}\oplus T^*_{1,0}$: viewing $H^{2,1}$ as a map $H^{2,1}:T_{1,0}X\to T^*_{0,1}X\otimes T^*_{1,0}X$, we define the holomorphic structure
\begin{equation}
\delbar_E = \mat{\delbar&0\\-H^{2,1}&\delbar}.
\end{equation}
The symmetric bilinear form on $E$ is as usual, i.e. $\IP{X+\xi,Y+\eta} = \tfrac{1}{2}(\xi(Y)+\eta(X))$, and the Courant bracket is twisted by the given 3-form:
\begin{equation}
[X+\xi,Y+\eta]_H = \mathcal{L}_X(Y+\eta) - i_Yd\xi + i_Yi_X H.
\end{equation}
This induces a holomorphic Courant bracket on the sheaf of $\delbar_E$--holomorphic sections.  

Furthermore, $E$ contains two holomorphic Dirac structures $D, F$, defined similarly to the real case described in Section~\ref{realdirac}:
\begin{equation}\label{FDdef}
\begin{aligned}
F & = \Tc M\times (T_{0,1}\CC\oplus T^*_{1,0}\CC),\\
D|_{M\times\{t\}} &= L_{\sigma_t}\times \Tc\CC.
\end{aligned}
\end{equation}
As in the real case, we obtain a Poisson structure by taking the Dirac sum: $D-F = L_{\sigma}$ defines the holomorphic Poisson structure on $X$ which is $\sigma_t$ on each fiber; the map $\pi$ is a Casimir, so that each fiber of $\pi$ is a Poisson submanifold of $X$.  Note that $F\cap D = T_{0,1} X$,  which means that $F$ and $D$ are complementary from the point of view of the holomorphic category (more precisely, after holomorphic reduction as described in \cite{MR3232003}, they become complementary holomorphic Dirac structures).  
Axiomatizing the above situation, we obtain the following.
\begin{definition}
A 1-parameter Hamiltonian deformation of holomorphic Poisson structures consists of the following:
\begin{enumerate}
\item A holomorphic submersion $p:X\to\CC$,
\item A holomorphic Courant algebroid $E$ on $X$ compatible with $p$ in the sense that it is equipped with a holomorphic Dirac structure $F\subset E$ with $\pi(F) = \ker{p_*}$, where $\pi:E\to \Tc X$ is the anchor. 
\item A holomorphic Dirac structure $D\subset E$ complementary to $F$.
\end{enumerate}
\end{definition}
The analogous result to Proposition~\ref{realhamasdir} is then
\begin{proposition}
The structures $(X,p,E,F,D)$ given above (\eqref{Hdef}--\eqref{FDdef}) define a 1-parameter Hamiltonian deformation of $(I,\sigma)$ if and only if $(I_t,\sigma_t,\beta_t)$ satisfies the system~\eqref{cxham}. 
\end{proposition}
We leave the detailed investigation of this result and its implications to future work; our only aim here was to explain how Hitchin's unobstructed deformations are in fact Hamiltonian families, controlled by a Dirac structure in a holomorphic Courant algebroid.  As in the real case, this implies that while the family of holomorphic Poisson structures may be non-isomorphic (even as complex manifolds), they remain in the same symplectic Morita equivalence class.

\section{Generalized K\"ahler deformations}\label{applic}

In order to use the unobstructedness result (Theorem~\ref{defex}) as an input to Theorem~\ref{diracdef} for the purpose of deforming generalized K\"ahler structures, we must impose that the family of closed 2-forms $\beta_t$ be real for all $t\in\RR$ in a neighbourhood of the origin.  The proof of Theorem~\ref{defex} carries through with minor changes, as long as we make the further assumption that the following composition is surjective (which is the case, for example, for K\"ahler manifolds):
\begin{equation}
H^2(M,\RR)\hookrightarrow H^2(M,\CC)\to H^2(M,\OO). 
\end{equation}

\begin{theorem}\label{defexreal}
Let $(M,I,\sigma)$ be a a compact holomorphic Poisson manifold for which the natural map $H^{2}(M,\RR)\to H^{2}(M,\OO)$ is surjective, and fix a real closed $(1,1)$-form $\omega_1\in \Omega^{1,1}(M,\RR)$ .
Then there is a family $F(t)$ of closed real 2-forms, analytic in $t$, for $t$ in a neighbourhood of $0\in\RR$, with $F(0)=0,$ $\del_tF(0) = \omega_1$, and with the property that $\omega(t) = \psi(F(t)) = (1+F(t)\sigma)^{-1}F(t)$ has vanishing $(0,2)$ part.

The corresponding Maurer-Cartan element $\sigma(\omega(t))$ defines a deformation of holomorphic Poisson structure $(I_t, \sigma_t)$ in which the complex structure $I_t$ varies with Kodaira-Spencer class $[\sigma (\del_t F)]\in H^1(T_{1,0}(I_t))$ at each time $t$ and the Poisson structure is given by the projection of $\sigma$ to its $(2,0)$ component relative to $I_t$.
\end{theorem}
\begin{proof}
The proof is similar to that for Theorem~\ref{defex} (using notation $F$ instead of $\beta$), inductively constructing $F = t\omega_1 + t^2 F_2 + t^3 F_3 + \cdots$, but with the following modifications.  We choose a splitting $\Hh\subset H^2(M,\RR)$ for the projection to $H^2(M,\OO)$.  Defining $r_{k+1}$ via $\psi(F_{\leq k})^{0,2} = t^{k+1}r_{k+1}\pmod t^{k+2}$,
we take $\wt r_{k+1}\in \Omega^2(M,\RR)$ be the unique harmonic lift of $r_{k+1}$ such that $[\wt r_{k+1}]\in\Hh$.  Then 
\begin{equation}
\wt r_{k+1}^{0,2} - r_{k+1} = \delbar \gamma_{k+1}
\end{equation} 
for a unique $\gamma_{k+1}\in\Omega^{0,1}$ in the image of $\overline{\partial}^*$.  We then define $F_{k+1} = -\wt r_{k+1} + d(\gamma_{k+1}+\ol\gamma_{k+1})$. The rest of the construction is the same as for Theorem~\ref{defex}.  

Since $F$ is real and we showed in Section~\ref{cxhamfam} that $\del_t F$ must have type $(2,0)+(1,1)$ relative to $I_t$, it follows that $\del_t F$ has type $(1,1)$ at all times.  By Proposition~\ref{hamfamcx}, we obtain the required description of the Kodaira-Spencer class, and finally we use~\eqref{newpoi} and the fact that $\rho=0$ for this deformation to obtain that $\sigma_t = \wedge^2 P_{1,0}(t)(\sigma_0)$, where $P_{1,0}(t)$ is the projection to $T_{1,0}(I_t)$, as required.
\end{proof}
Combining Theorems~\ref{defexreal} and \ref{diracdef}, we immediately obtain the following partial unobstructedness result for arbitrary generalized K\"ahler structures:
\begin{corollary}
Let $(M, \JJ_1(0),\JJ_2(0))$ be a compact generalized K\"ahler manifold with underlying holomorphic Poisson structures $(I_\pm,\sigma_\pm)$.  Assume $H^2(M,\RR)\to H^2(M,\OO)$ is surjective for the complex structure $I_+$.  Then for any closed $(1,1)$-form $\omega_1\in\Omega^{1,1}(M, I_+)$, there is a deformation of generalized K\"ahler structures $(\JJ_1(t), \JJ_2(t))$ for $t$ in a nonempty neighbourhood of $0\in\RR$, with the property that $(I_-,\sigma_-)$ remains fixed and $(I_+,\sigma_+)$ deforms with Kodaira-Spencer class $[\sigma_+(\omega_{1,1})]$.   
\end{corollary}
In addition to this result about deformations of an existing generalized K\"ahler structure, we may apply Theorem~\ref{defexreal} to construct generalized K\"ahler structures on any K\"ahler manifold equipped with a holomorphic Poisson structure, following Example~\ref{flocon}.  The existence of these structures is already known through the stability theorems of Goto~\cite{MR2669364,Goto:2009aa}; he actually used them to deduce the partial unobstructedness of the holomorphic Poisson structure.  Corollary~\ref{deformkahler} provides a complementary approach, using the unobstructedness of the Poisson structure to obtain the generalized  K\"ahler structures. 

%\begin{corollary}
%Let $(M, I,\sigma)$ be a compact holomorphic Poisson manifold and $\omega$ a K\"ahler form.  Then there is an analytic family $F(t)$ of closed real 2-forms, for $t$ in a nonempty neighbourhood of $0\in\RR$, with $F(0)=\omega$ and such that the pair of complex Dirac structures
%\begin{equation}
%L_1(t) = e^{iF(t)} L_{2it\sigma},\qquad L_2(t) = \Gamma_{iF(t)}   
%\end{equation}
%define a family of generalized K\"ahler structures $(\JJ_1(t), \JJ_2(t))$.  This family coincides with the given K\"ahler structure at $t=0$, and has the property that its underlying holomorphic Poisson structure $(I_+, \sigma_+)$ is given by $(I, t\sigma)$, while $(I_-,\sigma_-)$ undergoes a deformation given by 
%\begin{equation}
%L_{t\sigma_-(t)} = e^{F(t)}L_{t\sigma}
%\end{equation}
%with Kodaira-spencer class given by $[\sigma(\omega)]\in H^1(T)$. 
%\end{corollary}
\begin{corollary}\label{deformkahler}
Let $(M, I,\sigma)$ be a compact holomorphic Poisson manifold and $\omega$ a K\"ahler form.  Then there is an analytic family of generalized K\"ahler structures $(\JJ_1(t), \JJ_2(t))$, for $t$ in a nonempty neighbourhood of $0\in\RR$, which coincides with the given K\"ahler structure at $t=0$ and has the property that its underlying holomorphic Poisson structure $(I_+(t), \sigma_+(t))$ is given by $(I, t\sigma)$, while $(I_-(t),\sigma_-(t))$ is a deformation of $(I,\sigma)$ with Kodaira-Spencer class given by $[\sigma(\omega)]\in H^1(T)$. 
\end{corollary}
\begin{proof}
Theorem~\ref{defexreal} provides a family $\wt F(t) = t\omega + O(t^2)$ of real closed 2-forms such that $L_{\sigma_t}=e^{\wt F(t)}L_\sigma$ defines a family of Poisson structures with the required Kodaira-Spencer class. Scaling the Dirac structures by $t^{-1}$ and defining $F(t) = t^{-1}\wt F(t)$, we obtain the relation
\begin{equation}
L_{t\sigma_-(t)} = e^{F(t)}L_{t\sigma}.
\end{equation}
Following Proposition~\ref{defgk}, we define $F_- = F$ and $F_+ = B = 0$ and obtain the complex Dirac structures
\begin{equation}
L_1(t) = e^{iF(t)}L_{t\sigma/2i},\qquad L_2(t) = e^{iF(t)}\Tc = \Gamma_{iF(t)},
\end{equation} 
we obtain the required generalized K\"ahler deformation of $(I,\omega)$.
\end{proof}

\bibliographystyle{hyperamsplain} 
\bibliography{references}

\end{document}

%% file: preamble.tex
\usepackage[width=15cm,top = 3.5cm,bottom=3.5cm]{geometry}
\usepackage{amsthm,amsfonts,amssymb,amsmath,amscd} % AMS packages
\usepackage[all,cmtip,pdf]{xy}
\usepackage{aliascnt} % Extended theorem numbering
\usepackage{mathrsfs} % Script font
\usepackage{xargs,ifthen} % Macro construction utilities
\usepackage{comment}
\usepackage{color}
\usepackage{tikz}
\usetikzlibrary{calc}

\makeatletter
\def\@cite#1#2{{\normalfont[{#1\if@tempswa , #2\fi}]}}
\makeatother

\usepackage{hyperref}
\definecolor{tocolor}{rgb}{.1,.1,.1}
\definecolor{urlcolor}{rgb}{.2,.2,.6}
\definecolor{linkcolor}{rgb}{.1,.1,.5}
\definecolor{citecolor}{rgb}{.4,.2,.1}
\hypersetup{colorlinks=true, urlcolor=urlcolor, linkcolor=linkcolor, citecolor=citecolor}

\newcommandx{\thdef}[2]{
	\newaliascnt{#1}{theorem}  
	\newtheorem{#1}[#1]{#2}
	\aliascntresetthe{#1}  
	\newtheorem*{#1*}{#2}
	\expandafter\newcommand\expandafter{\csname #1autorefname\endcsname}{#2}
}

\makeatletter
\newtheorem*{rep@theorem}{\rep@title}
\newcommand{\newreptheorem}[2]{%
\newenvironment{rep#1}[1]{%
 \def\rep@title{#2 \ref{##1}}%
 \begin{rep@theorem}}%
 {\end{rep@theorem}}}
\makeatother

\newtheorem{theorem}{Theorem}[section]
\newreptheorem{theorem}{Theorem}

\thdef{lemma}{Lemma}
\thdef{corollary}{Corollary}
\thdef{conjecture}{Conjecture}
\newreptheorem{conjecture}{Conjecture}
\thdef{proposition}{Proposition}

\theoremstyle{definition}
\thdef{definition}{Definition}
\thdef{notation}{Notation}

\theoremstyle{remark}
\thdef{remark}{Remark}
\thdef{remarks}{Remarks}

\theoremstyle{remark}
\thdef{ex}{Example}

\newenvironment{example}
{\begin{ex}}%
{\hfill $\blacksquare$\end{ex}}

\newtheoremstyle{parag}
  {\topsep}   % ABOVESPACE
  {\topsep}   % BELOWSPACE
  {}  % BODYFONT
  {}       % INDENT (empty value is the same as 0pt)
  {\bfseries} % HEADFONT
  {.}         % HEADPUNCT
  { } % HEADSPACE
  {}          % CUSTOM-HEAD-SPEC
\theoremstyle{parag}

\newcommand{\spc}[1]{\mathsf{#1}} % For spaces
\newcommand{\shf}[1]{\mathcal{#1}} % For sheaves

\newcommand{\RR}{\mathbb{R}}
\newcommand{\CC}{\mathbb{C}}

\newcommand{\HH}{\mathbb{H}}

\newcommand{\rbrac}[1]{\left(#1\right)} % Round brackets
\newcommand{\IP}[1]{\left\langle#1\right\rangle} % Angle brackets

\newcommandx{\fn}[2][2=]{#1\ifthenelse{\equal{#2}{}}{}{\!\rbrac{{#2}}}} % Functions with optional arguments
\newcommandx{\id}[2][2=]{\fn{{\rm id}_{#1}}[#2]} % The identity function \id{X} = id_X

\newcommand{\ext}[2][\bullet]{\spc{\Lambda}^{#1}{#2}} % exterior products
\newcommandx{\End}[2][1=]{\fn{\spc{End}_{#1}}[#2]} % endomorphisms
\newcommandx{\Hom}[2][1=]{\fn{\spc{Hom}_{#1}}[#2]} % homs
\newcommandx{\Aut}[2][1=]{\fn{\spc{Aut}_{#1}}[#2]} % automorphisms
\newcommandx{\image}[1]{\fn{\spc{img}}[#1]} % image
\renewcommandx{\ker}[1]{\fn{\spc{ker}}[#1]} % kernel
\newcommandx{\rank}[1]{\fn{\mathrm{rank}}[#1]} % rank

\newcommandx{\ann}[1]{\fn{\spc{ann}}[#1]} % annihilator of a subspace
\newcommandx{\hlgy}[3][1=\bullet,3=]{\spc{H}_{#1}^{#3}\!\rbrac{{#2}}} % Homology
\newcommandx{\cohlgy}[3][1=\bullet,3=]{\spc{H}^{#1}_{#3}\!\rbrac{{#2}}} % Cohomology
\newcommandx{\chow}[3][1=\bullet,3=]{\spc{A}^{#1}_{#3}\!\rbrac{{#2}}} % Chow ring

\newcommandx{\Ext}[3][1=\bullet,3=]{\fn{\spc{Ext}^{#1}_{#3}}[{#2}]} % Ext group
\newcommandx{\Tor}[3][1=\bullet,3=]{\fn{\spc{Tor}^{#1}_{#3}}[{#2}]} % Tor group

\newcommandx{\Pic}[1]{\fn{\spc{Pic}}[{#1}]} % Picard group
\newcommandx{\chernalg}[2][1=\bullet]{\fn{\spc{Chern}^{#1}}[{#2}]} % Chern algebra
\newcommandx{\chern}[2][1=]{\fn{c_{#1}}[#2]} % Chern class
\newcommandx{\ch}[2][1=]{\fn{\mathrm{ch}_{#1}}[{#2}]} % Chern character

\newcommandx{\sKer}[2][1=]{ \fn{ \shf{K}er_{#1}}[{#2}] } % Kernel sheaf
\newcommandx{\sHom}[2][1=]{ \fn{ \shf{H}om_{#1}}[{#2}] } % Hom sheaf
\newcommandx{\sEnd}[2][1=]{ \fn{ \shf{E}nd_{#1}}[{#2}] } % End sheaf
\newcommandx{\sExt}[3][1=\bullet,3=]{\fn{\shf{E}xt^{#1}_{#3}}[{#2}]} % Ext sheaf
\newcommandx{\sTor}[3][1=\bullet,3=]{\fn{\shf{T}or^{#1}_{#3}}[{#2}]} % Tor sheaf

\newcommandx{\forms}[2][1=\bullet]{\Omega^{#1}_{#2}} % Differential forms
\newcommandx{\can}[1][1=]{\omega_{#1}} % canonical bundle
\newcommandx{\acan}[1][1=]{\omega_{#1}^{-1}} % anti-canonical bundle
\newcommandx{\tshf}[1]{\shf{T}_{#1}} % Tangent sheaves
\newcommandx{\mvect}[2][1=\bullet]{ \ext[#1]{\tshf{#2}} }% Multivector fields
\newcommandx{\der}[2][1=\bullet]{\mathscr{X}^{#1}_{#2}} % Multivector fields
\newcommandx{\sJet}[3][1=,2=]{\shf{J}^{#1}_{#2}#3} % Jet sheaf

\newcommandx{\tb}[2][1=]{\spc{T}_{\!#1}{#2}} % Tangent bundle
\newcommandx{\ctb}[2][1=]{\spc{T}_{\!#1}^*{#2}} % Cotangent bundle
\newcommandx{\lie}[2][2=]{\fn{\mathscr{L}_{#1}}[#2]} % Lie derivative
\newcommandx{\hook}[2][2=]{\fn{i_{#1}}[#2]} % Interior product
\newcommand{\del}{\partial}

\newcommand{\X}{\spc{X}}

\makeatletter
\newcommand{\thickbar}{\mathpalette\@thickbar}
\newcommand{\@thickbar}[2]{{#1\mkern1.5mu\vbox{
  \sbox\z@{$#1\mkern-1mu#2\mkern-1mu$}%
  \sbox\tw@{$#1\overline{#2}$}%
  \dimen@=\dimexpr\ht\tw@-\ht\z@-.6\p@\relax
  \hrule\@height.4\p@ % adjust for the desired rule thickness
  \vskip1\p@
  \hrule\@height.4\p@ % adjust for the desired rule thickness
  \vskip\dimen@
  \box\z@}\mkern1.5mu}
}
\makeatother